\documentclass[12pt,a4paper,usenames,dvipsnames]{article}

\usepackage[utf8]{inputenc}
\usepackage{fullpage}
\usepackage{psfrag}
\usepackage{amsmath}
\usepackage{fancybox}
\usepackage{amsfonts}
\usepackage{amsthm}
\usepackage{amssymb}
\usepackage{comment}
\usepackage{graphicx}

\usepackage{booktabs}

\usepackage{caption}
\usepackage{subcaption}
\usepackage{geometry}
\renewcommand{\arraystretch}{1.3} 

\usepackage[normalem]{ulem} 
\usepackage{tikz}
\usetikzlibrary{decorations.pathreplacing}
\usepackage{enumerate}
\usepackage{MnSymbol}
\usepackage{bm}
\usepackage{multirow}
\usepackage{colortbl}

\usepackage{caption}
\usepackage{algorithm}
\usepackage{setspace}
\usepackage{algpseudocode}

\usepackage{booktabs}
\usepackage{longtable}
\usepackage{colortbl}
\usepackage{multirow}

\usepackage{pifont}
 \usepackage[colorinlistoftodos,prependcaption,spanish,textsize=footnotesize]{todonotes}
 
\newcounter{todoListItems} 
\newcommand{\todoTrans}[2][ ]{
\addtocounter{todoListItems}{1}
\todo[#1]
{#2 \hfill \hyperlink{todo\thetodoListItems}{$\uparrow$}}}

 \usepackage{xargs} 
 
\newcommandx{\NanotodoPB}[2][1=]{\todoTrans[caption={\protect\hypertarget{todo\thetodoListItems}{} #1},
inline,linecolor=OliveGreen,backgroundcolor=OliveGreen!25,bordercolor=OliveGreen]{#2}}

\newcommandx{\NanotodoPM}[2][1=]{\todoTrans[caption={\protect\hypertarget{todo\thetodoListItems}{} #1},
inline,linecolor=orange,backgroundcolor=orange!25,bordercolor=orange]{#2}}

\newcommandx{\NanotodoPA}[2][1=]{\todoTrans[caption={\protect\hypertarget{todo\thetodoListItems}{} #1},
inline,linecolor=red,backgroundcolor=red!25,bordercolor=red]{#2}}

\newcommandx{\Edutodo}[2][1=]{\todoTrans[caption={\protect\hypertarget{todo\thetodoListItems}{} #1},
inline,linecolor=blue,backgroundcolor=blue!25,bordercolor=blue]{#2}}

\newcommandx{\Marcetodo}[2][1=]
{\todoTrans[caption={\protect\hypertarget{todo\thetodoListItems}{} #1},
inline,linecolor=Plum,backgroundcolor=Plum!25,bordercolor=Plum]{#2}}

 \theoremstyle{plain}
    \newtheorem{theorem}{Theorem}[section]

    \newtheorem{proposition}[theorem]{Proposition}
    
 \theoremstyle{definition}
    
    \newtheorem{remark}[theorem]{Remark}
    
    \newtheorem{procedure}[theorem]{Procedure}
 \theoremstyle{remark}

 \usepackage[english]{babel}

\usepackage[%
   bookmarks=true,
   pagebackref=false, 
   colorlinks,
   linkcolor=blue,
   anchorcolor=blue,
   citecolor=blue,
   filecolor=blue,
   urlcolor=blue
   ]{hyperref}

\usepackage{bm}
\usepackage{xspace}

\usepackage{color}
\definecolor{miverde}{RGB}{128,255,0}
\definecolor{minaranja}{RGB}{255,128,1}
\definecolor{gris}{RGB}{128, 128, 128}

\DeclareMathOperator{\vect}{vec}

\newcommand\NN{\mathbb N}

\newcommand\BB{\mathcal B}

\newcommand{\RR}{\mathbb R}
\newcommand{\h}{h}

\renewcommand{\Xi}{{\bm\xi}}

\renewcommand{\nu}{{\bm\nu}}

\usepackage{authblk}
\newcommand{\norm}[1]{{\left\vert\kern-0.25ex\left\vert\kern-0.25ex\left\vert #1 
    \right\vert\kern-0.25ex\right\vert\kern-0.25ex\right\vert}}

\newcommand\SSS{\mathcal S}

\begin{document}

\title{\bf Left Inverses for B-spline Subdivision Matrices in Tensor-Product Spaces}
\author{Marcelo Actis, Silvano Figueroa \& Eduardo M. Garau\thanks{Corresponding author: Eduardo M. Garau (egarau@santafe-conicet.gov.ar)}}
	\affil{\small Universidad Nacional del Litoral, \\
 Consejo Nacional de Investigaciones Cient\'ificas y T\'ecnicas, \\
 FIQ, Santa Fe, Argentina}


\maketitle

\begin{abstract}
	In this article, we study dyadic coarsening operators in univariate spline spaces and in tensor-product spline spaces over uniform grids. 
	Our construction is strongly motivated by the work of Bartels, Golub, and Samavati (2006), \emph{Some observations on local least squares}, \textit{BIT}, 46(3):455–477. 
	The proposed operators are local in nature and yield approximations to a given spline that are comparable to the global \(L^2\)-best approximation, while being significantly faster to compute and computationally inexpensive.
\end{abstract}


\begin{quote}\small
\end{quote}




\section{Introduction}\label{Introduccion}

When we consider a spline space of maximum smoothness over a uniform partition, every function in this space also belongs to the larger spline space defined over the refined partition obtained by inserting all midpoints. Lane and Riesenfeld~\cite{LR1980} introduced a geometric approach to compute the B-spline coefficients in the refined space from the coefficients in the coarse space. Algebraically, this process corresponds to applying the uniform knot insertion matrix between the two spaces, which we refer to as the \emph{subdivision matrix}. Structurally, the subdivision matrix exhibits a banded sparsity pattern, reflecting the compact support of B-spline basis functions. The non-zero entries in each column represent the coefficients of the linear combination of fine-level B-splines that reconstruct each coarse-level B-spline in the refined space.

On the other hand, left inverses of subdivision matrices provide a way to transfer coefficients from fine B-splines to coarse ones. The pattern—specifically, the number and position of non-zero elements in each row—indicates which fine B-splines are combined through a precise linear relation to construct the corresponding coefficient of a coarse function.

In this article we use the local least-squares technique introduced in~\cite{BGS2006-local-least-squares} to derive explicit left-inverse formulas for the subdivision matrix of B-splines. These left inverses are band matrices and enable us to obtain a B-spline representation in a coarse space that is comparable in quality to the one obtained by solving a global least-squares system. Each column of these matrices contains the coefficients that transfer information from a fine B-spline to a small set of nearby ancestors at the coarsest level. In~\cite{BGS2006-local-least-squares}, the properties of the proposed method for B-spline subdivision matrices were examined only through numerical tests. The aim of the present paper is to further develop and strengthen this analysis.

In~\cite{BS2000} the problem of constructing a left inverse for subdivision matrices was addressed in the context of local multiresolution filters for quadratic and cubic B-splines. A relevant and comparable aspect to our objective is that the proposed method reduces the problem to local cases by exploiting the specific structure of the subdivision matrices. In general, this leads to underdetermined systems of equations, whose resolution is purely algebraic, selecting the solution that minimizes the Euclidean norm among all possible ones. On the other hand, in~\cite{BRSF2011}, the same goal is pursued, but with a significant methodological improvement based on singular value decomposition (SVD), which allows for a more robust and stable approach to the problem. However, in both studies the determination of local problem sizes is context-dependent, lacking a generalizable selection strategy.

In~\cite{Brenna2002} two different approaches to knot removal are presented, together with strategies for finding suitable approximations. For both methods, the approximations are carried out using the algorithm proposed in~\cite{LM87}, extended to tensor-product splines. The methods differ in how they apply the knot removal process. The first approach addresses all parametric directions of a tensor-product spline simultaneously, while the second treats one parametric direction at a time.

Finally, we note that a rigorous error analysis of the coarsening process in the univariate spline case has been carried out in~\cite{FGM24}. We also emphasize that the analysis developed in the present article provides a foundation for the design and study of coarsening strategies in multivariate hierarchical spline spaces~\cite{Kraft, GJS14}. These spaces form a highly suitable framework for adaptive methods in partial differential equations as well as for function and data approximation. More specifically, hierarchical spline spaces possess a multilevel structure: the initial level is a tensor-product spline space over a uniform mesh, and each subsequent level is obtained by dyadic refinement of the preceding one. The B-splines at each level are supported on specific subregions of the domain, making the design of coarsening operators between levels—particularly local ones—both natural and advantageous.

This article is organized as follows. In Section~\ref{sec:b_splines_and_subdivision_matrices}, we introduce the framework and recall the basic concepts required throughout the paper. Section~\ref{sec:left_inverses} is devoted to the construction of coarsening operators and to an exploration of their structure. In Section~\ref{sec:tensor product}, we briefly review the construction of tensor-product spline spaces and extend the univariate coarsening operators to this setting. Finally, in Section~\ref{sec:tests}, we examine the stability and accuracy of the proposed operators and illustrate their performance through numerical experiments.

\section{B-splines and subdivision matrices}\label{sec:b_splines_and_subdivision_matrices}

Let \( p \in\NN\) be fixed, and let \( \hat{\mathcal{S}} \) be the space of polynomial splines of degree at most $p$ with maximum smoothness over the partition \( \hat{Z} \) of a given interval $[a,b]\subset\RR$, into \( \hat{N} \) equidistributed breakpoints, namely  
\[
\hat{Z} = \{a= \hat{\zeta}_1 , \hat{\zeta}_2, \dots, \hat{\zeta}_{\hat{N}-1}, \hat{\zeta}_{\hat{N}} = b \}.
\]
Let \( \hat{\mathcal{B}} = \{ \hat{\beta}_1,\dots,\hat{\beta}_{\hat{n}}\}\) be the B-spline basis (cf.~\cite{DeBoor, Schumi}) associated with the \((p+1)\)-open knot vector \( \hat{\Xi} \) given by 
\[
\hat{\bm{\xi}} = \{\hat{\xi}_1, \dots, \hat{\xi}_{\hat{n}+p+1}\}= \{ \underbrace{a, \dots, a}_{p+1}, \hat{\xi}_{p+2}, \hat{\xi}_{p+3}, \dots, \hat{\xi}_{\hat{n}}, \underbrace{b, \dots, b}_{p+1} \},
\]
where  
\(
\hat{n} = p+\hat{N}-1\), and \(\hat{\xi}_i = \hat{\zeta}_{i-p}\), for \(p+2 \leq i \leq \hat{n}.
\)
Unlike several works in the literature that restrict attention to uniform partitions, here we consider B-spline bases associated with open knot vectors. This choice is particularly convenient for imposing boundary conditions, both in quasi-interpolation and in the numerical solution of differential equations.

Next, we refine the partition \(\hat{Z}\) by inserting the midpoint of each subinterval, obtaining a new partition \( Z \):  
\[
Z = \{\zeta_1,\dots,\zeta_N\}=\{ \hat\zeta_1 , m_1, \hat\zeta_2, m_2, \dots, \hat\zeta_{\hat{N}-1}, m_{\hat{N}-1}, \hat\zeta_{\hat{N}} \},
\]
where the midpoints are defined by  
\(
m_i = \tfrac{\hat{\zeta}_{i} + \hat{\zeta}_{i+1}}{2},\) for  \(1 \leq i \leq \hat{N}-1.
\)
The refined partition \( Z \) then consists of \( N = 2\hat{N} - 1 \) breakpoints. The odd-indexed breakpoints correspond to the original partition,  
\[
\zeta_{2j-1} = \hat{\zeta}_j, \quad 1 \leq j \leq \hat{N},
\]
while the even-indexed breakpoints are the newly inserted midpoints:  
\[
\zeta_{2j} = m_j, \quad 1 \leq j \leq \hat{N} - 1.
\]
The spline space \( \mathcal{S} \) of maximum smoothness over \( Z \) has dimension  
\(
n = p+N-1.
\)
Let \( \mathcal{B} :=\{\beta_1,\dots,\beta_n\}\) be the corresponding B-spline basis associated with the $(p+1)$-open knot vector \( \bm{\xi} \), given by 
\[
\bm{\xi} = \{\xi_1, \dots, \xi_{n+p+1}\}= \{ \underbrace{a, \dots, a}_{p+1}, \xi_{p+2}, \xi_{p+3}, \dots, \xi_{n}, \underbrace{b, \dots, b}_{p+1} \},
\]
where the internal knots are  
\(
\xi_i = \zeta_{i-p}\), for \(p+2 \leq i \leq n.
\)
By construction, the knot vector \( \bm{\xi} \) refines \( \hat{\bm{\xi}} \), and its internal knots satisfy 
\[
\xi_{p+2j-1} = \hat{\xi}_{p+j}, \quad 1 \leq j \leq \hat{N}, \quad \text{and} \quad \xi_{p+2j} = m_{j}, \quad 1 \leq j \leq \hat{N}-1.
\]
Thus, the refined knot vector \( \bm{\xi} \) contains all the knots of \( \hat{\bm{\xi}} \), together with additional midpoints (cf. Figure~\ref{F:partitions}), which yield a higher-resolution spline space \( \mathcal{S} \). In other words, \(\hat{\mathcal{S}}\subset\mathcal{S}\).

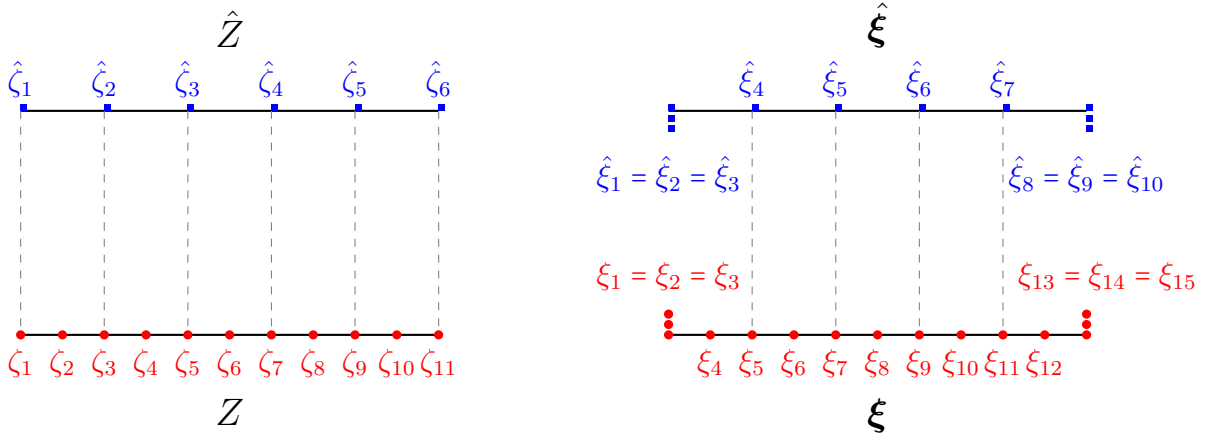
\begin{figure}[htbp]
\begin{tikzpicture}[scale=5.5]

    \def\Nhat{6}
    \def\p{2}
    \pgfmathsetmacro{\h}{1/(\Nhat-1)}
    \pgfmathsetmacro{\N}{2*(\Nhat-1)+1}
    \pgfmathsetmacro{\hf}{1/(\N-1)}
    \def\sep{0.27} 
    \def\xshift{1.55}
    \def\extsep{0.025} 
    \def\markerSize{0.015}
    \def\markerOffset{0.015} 
  
    \newcommand{\squaremarker}[2]{
      \fill[blue] (#1,#2+\markerOffset) rectangle ++(\markerSize,-\markerSize);
    }
  
    \newcommand{\circlemarker}[2]{
      \fill[red] (#1,#2) circle[radius=0.011];
    }
  
    \begin{scope}
  
      \begin{scope}[yshift=\sep cm]
        \draw[thick] (0,0) -- (1,0);
        \node[above=20pt,font=\large] at (0.5,0) {$\hat Z$};
  
        \foreach \i in {1,...,\Nhat} {
          \pgfmathsetmacro{\x}{(\i-1)*\h}
          \squaremarker{\x}{0}
          \node[above=1pt,blue] at (\x,0) {$\hat\zeta_{\i}$};
        }
\foreach \i in {1,...,6} {
    \pgfmathsetmacro{\x}{(\i-1)*\h}
    \draw[dashed,gray] (\x,0) -- (\x,-\sep);
  }

      \end{scope}
  
      \begin{scope}[yshift=-\sep cm]
        \draw[thick] (0,0) -- (1,0);
        \node[below=20pt,font=\large] at (0.5,0) {$Z$};
  
        \foreach \j in {1,...,\N} {
          \pgfmathsetmacro{\x}{(\j-1)*\hf}
          \circlemarker{\x}{0}
          \node[below=1pt,red] at (\x,0) {$\zeta_{\j}$};
        }
  
        \foreach \i in {1,...,\Nhat} {
          \pgfmathsetmacro{\x}{(\i-1)*\h}
          \draw[dashed,gray] (\x,\sep) -- (\x,0);
        }

      \end{scope}
  
    \end{scope}
  
    \begin{scope}[xshift=\xshift cm]
  
      \begin{scope}[yshift=\sep cm]
        \draw[thick] (0,0) -- (1,0);
        \node[above=20pt,font=\large] at (0.5,0) {$\hat\Xi$};
  
        \foreach \i in {0,1,2} {
          \pgfmathsetmacro{\dy}{-\i*\extsep}
          \squaremarker{0}{\dy}
        }
        \node[below=5pt,blue] at (0,-0.05) {$\hat\xi_1 = \hat\xi_2 = \hat\xi_3$};
  
        \foreach \i in {1,...,4} {
          \pgfmathsetmacro{\x}{\i*\h}
          \pgfmathtruncatemacro{\lab}{\p + \i + 1}
          \squaremarker{\x}{0}
          \node[above=1pt,blue] at (\x,0) {$\hat\xi_{\lab}$};
        }
  
        \foreach \i in {0,1,2} {
          \pgfmathsetmacro{\dy}{-\i*\extsep}
          \squaremarker{1}{\dy}
        }
        \node[below=5pt,blue] at (1,-0.05) {$\hat\xi_8 = \hat\xi_9 = \hat\xi_{10}$};
  
        \foreach \i in {2,...,5} {
          \pgfmathsetmacro{\x}{(\i-1)*\h}
          \draw[dashed,gray] (\x,0) -- (\x,-\sep);
        }
      \end{scope}
  
      \begin{scope}[yshift=-\sep cm]
        \draw[thick] (0,0) -- (1,0);
        \node[below=20pt,font=\large] at (0.5,0) {$\Xi$};
  
        \foreach \i in {0,1,2} {
          \pgfmathsetmacro{\dy}{\i*\extsep}
          \circlemarker{0}{\dy}
        }
        \node[above=5pt,red] at (0,0.05) {$\xi_1 = \xi_2 = \xi_3$};
  
        \foreach \j in {1,...,9} {
          \pgfmathsetmacro{\x}{\j*\hf}
          \pgfmathtruncatemacro{\lab}{\p + \j + 1}
          \circlemarker{\x}{0}
          \node[below=1pt,red] at (\x,0) {$\xi_{\lab}$};
        }
  
        \foreach \i in {0,1,2} {
          \pgfmathsetmacro{\dy}{\i*\extsep}
          \circlemarker{1}{\dy}
        }
        \node[above=5pt,red] at (1.05,0.05) {$\xi_{13} = \xi_{14} = \xi_{15}$};
  
        \foreach \i in {2,...,5} {
          \pgfmathsetmacro{\x}{(\i-1)*\h}
          \draw[dashed,gray] (\x,\sep) -- (\x,0);
        }
      \end{scope}
  
    \end{scope}
  
  \end{tikzpicture}
  \caption{Uniform partition $\hat Z$ with $\hat N = 6$ elements and its dyadic refinement $Z$.
The associated open knot vectors $\hat\Xi$ and $\Xi$ correspond to polynomial degree $p = 2$,
with multiplicity $p+1 = 3$ at the endpoints of the interval.}\label{F:partitions}
\end{figure}

Refining a knot vector by midpoint insertion provides a way to increase the degrees of freedom (DOFs) of a spline function without altering its shape. This refinement enables more detailed variations to be captured by modifying its B-spline coefficients. 

In general, the basis functions in \(\BB\) are related to those in \(\hat{\BB}\) through a linear transformation:
\begin{equation*}
	\hat\beta_{j}(x) = \sum_{i=1}^{n} A_{ij} {\beta}_{i}(x), \qquad j =1, \ldots, \hat n.   
\end{equation*}
This formula establishes the classical two-scale relation between coarse and fine B-splines, where the coefficients \(A_{ij}\) form the so-called \emph{subdivision matrix} \(A\in \mathbb{R}^{n \times \hat{n}}\). In compact form,
\begin{equation}\label{E:subdivision matrix basis functions} \hat{\boldsymbol{\beta}}(x) = A^{T} \, \boldsymbol{\beta}(x), \end{equation}
with
\[ \boldsymbol{\beta}(x) := 
\begin{bmatrix}
	\beta_1(x) \\
	\beta_2(x) \\
	\vdots \\
	\beta_n(x)
\end{bmatrix},
\qquad
\hat{\boldsymbol{\beta}}(x) := 
\begin{bmatrix}
	\hat{\beta}_1(x) \\
	\hat{\beta}_2(x) \\
	\vdots \\
	\hat{\beta}_{\hat{n}}(x)
\end{bmatrix}.\] 

In particular, each coarse B-spline \( \hat{\beta}_j \) can be expressed as a simple linear combination of fine B-splines (cf. Figure~\ref{F:two_scale}):  
\[ \hat{\beta}_j = \sum_{i=1}^{p+2} \eta_{i,p} \beta_{2j-p-2+i}, \qquad p+1 \leq j \leq \hat{n}-p. \]
Here, the indices in the sum are chosen so that the support of \( \hat{\beta}_j \), which spans \( p+1 \) coarse knot intervals, coincides exactly with the support of \( p+2 \) consecutive fine B-splines. The coefficient vector \(\boldsymbol{\eta}_p:=(\eta_{i,p})\) is given by  
\[ \eta_{i,p} = 2^{-p} \binom{p+1}{i-1}, \quad 1 \leq i \leq p+2. \]

\begin{figure}[htbp]
    \centering
    \includegraphics[width=\linewidth]{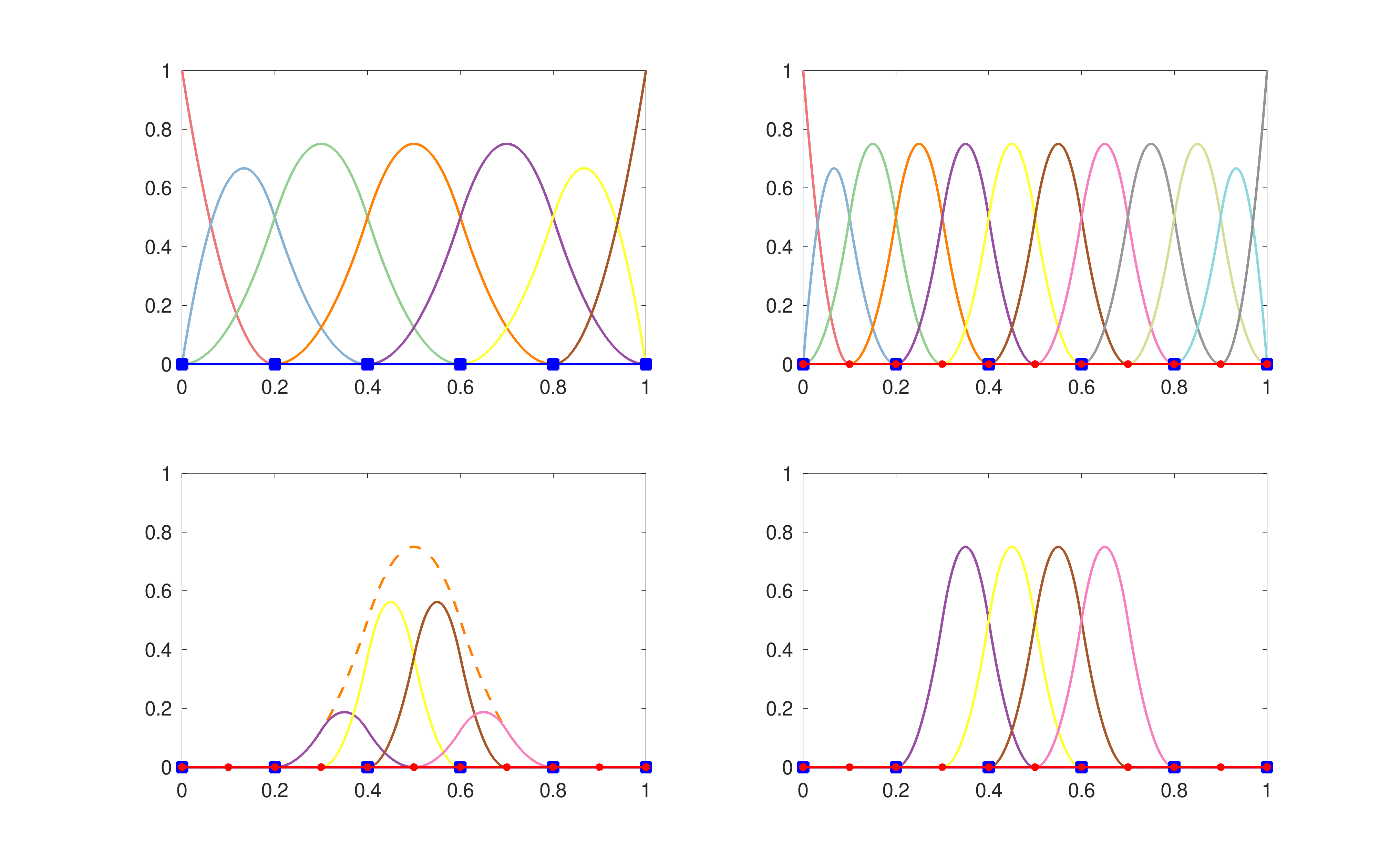}
    \caption{Coarse and fine quadratic B-spline bases are shown at the top left and top right, respectively. At the bottom left, a coarse B-spline basis function (in dashed line) can be expressed as a linear combination of the fine B-spline basis functions shown at the bottom right.
    }\label{F:two_scale}
\end{figure}

Moreover, any spline $\hat{s} \in \hat{\SSS}$ can be written as a linear combination of coarse basis functions:
\begin{equation*}
	\hat{s} = \sum_{i=1}^{\hat{n}} \hat{c}_i \hat{\beta}_i={\bf \hat{c}}^T \hat{\boldsymbol{\beta}},
\end{equation*}
with ${\bf \hat{c}} = (\hat{c}_1, \dots, \hat{c}_{\hat{n}})^T \in \mathbb{R}^{\hat{n}}$. Since $\hat{\SSS} \subset \SSS$, the same function can also be expressed in terms of the fine basis:
\begin{equation*}
	\hat{s} = \sum_{i=1}^{n} c_i \beta_i={\bf {c}}^T {\boldsymbol{\beta}},
\end{equation*}
with ${\bf c} = (c_1, \dots, c_n)^T \in \mathbb{R}^n$. 

Now, taking into account~\eqref{E:subdivision matrix basis functions}, the fine B-spline coefficients are obtained from the coarse ones using the subdivision matrix:
\begin{equation} \label{knot insertion 1D}
	{\bf c} = A {\bf \hat{c}}.    
\end{equation}
Figure~\ref{F:knot insertion matrix} illustrates several examples of subdivision matrices for different values of $p$.

\begin{figure}[htbp]
  \scriptsize
\[
A_1 =\renewcommand{\arraystretch}{1.3}
\begin{bmatrix}
1 & & & & & & \\
\frac{1}{2} & \frac{1}{2} & & & & & \\
  & 1 & & & & & \\
  & \frac{1}{2} & \frac{1}{2} & & & & \\
  &  & 1 & & & & \\
  &  & \frac{1}{2} & & & & \\
  &  &  & \ddots & & & \\
  &  &  & & \frac{1}{2} & & \\
  &  &  &  & 1 & & \\
  &  &  &  & \frac{1}{2} & \frac{1}{2} & \\
  &  &  &  &  & 1 & \\
  &  &  &  &  & \frac{1}{2} & \frac{1}{2} \\
  &  &  &  &  &  & 1
\end{bmatrix}
\qquad\qquad\qquad
A_3= \left[\renewcommand{\arraystretch}{1.5}
\begin{array}{*{11}{c}}
1 & & & & & & & & & & \\
\frac{1}{2} & \frac{1}{2} & & & & & & & & & \\
& \frac{3}{4} & \frac{1}{4} & & & & & & & & \\
& \frac{3}{16} & \frac{11}{16} & \frac{1}{8} & & & & & & & \\
& & \frac{1}{2} & \frac{1}{2} & & & & & & & \\
& & \frac{1}{8} & \frac{3}{4} & \frac{1}{8} & & & & & & \\
& & & \frac{1}{2} & \frac{1}{2} & & & & & & \\
& & & \frac{1}{8} & \frac{3}{4} & & & & & & \\
& & & & \frac{1}{2} & & & & & & \\
& & & & \frac{1}{8} & & & & & & \\
& & & &  & \ddots& & & & & \\
& & & & & & \frac{1}{8} & & & & \\
& & & & & & \frac{1}{2} & & & & \\
& & & & & & \frac{3}{4} & \frac{1}{8} & & & \\
& & & & & & \frac{1}{2} & \frac{1}{2} & & & \\
& & & & & & \frac{1}{8} & \frac{3}{4} & \frac{1}{8} & & \\
& & & & & & & \frac{1}{2} & \frac{1}{2} & & \\
& & & & & & & \frac{1}{8} & \frac{11}{16} & \frac{3}{16} & \\
& & & & & & & & \frac{1}{4} & \frac{3}{4} & \\
& & & & & & & & & \frac{1}{2} & \frac{1}{2} \\
& & & & & & & & & & 1
\end{array}\right]
\]
\[
A_2 =\renewcommand{\arraystretch}{1.5}
\begin{bmatrix}
1 &  &  &  &  &  &  &  &  \\
\frac{1}{2} & \frac{1}{2} &  &  &  &  &  &  &  \\
  & \frac{3}{4} & \frac{1}{4} &  &  &  &  &  &  \\
  & \frac{1}{4} & \frac{3}{4} &  &  &  &  &  &  \\
  &  & \frac{3}{4} & \frac{1}{4} &  &  &  &  &  \\
  &  & \frac{1}{4} & \frac{3}{4} &  &  &  &  &  \\
  &  &  & \frac{3}{4} &  &  &  &  &  \\
  &  &  & \frac{1}{4} &  &  &  &  &  \\
  &  &  &  & \ddots &  &  &  &  \\
  &  &  &  &  & \frac{1}{4} &  &  &  \\
  &  &  &  &  & \frac{3}{4} &  &  &  \\
  &  &  &  &  & \frac{3}{4} & \frac{1}{4} &  &  \\
  &  &  &  &  & \frac{1}{4} & \frac{3}{4} &  &  \\
  &  &  &  &  &  & \frac{3}{4} & \frac{1}{4} &  \\
  &  &  &  &  &  & \frac{1}{4} & \frac{3}{4} &  \\
  &  &  &  &  &  &  & \frac{1}{2} & \frac{1}{2} \\
  &  &  &  &  &  &  &  & 1
\end{bmatrix}
\quad
A_4=\left[\renewcommand{\arraystretch}{1.5}
\begin{array}{*{13}{c}}
1 & & & & & & & & & & & & \\
\frac{1}{2} & \frac{1}{2} & & & & & & & & & & & \\
& \frac{3}{4} & \frac{1}{4} & & & & & & & & & & \\
& \frac{3}{16} & \frac{11}{16} & \frac{1}{8} & & & & & & & & & \\
& & \frac{5}{12} & \frac{25}{48} & \frac{1}{16} & & & & & & & & \\
& & \frac{1}{12} & \frac{29}{48} & \frac{5}{16} & & & & & & & & \\
& & & \frac{5}{16} & \frac{5}{8} & \frac{1}{16} & & & & & & & \\
& & & \frac{1}{16} & \frac{5}{8} & \frac{5}{16} & & & & & & & \\
& & & & \frac{5}{16} & \frac{5}{8} & & & & & & & \\
& & & & \frac{1}{16} & \frac{5}{8} & & & & & & & \\
& & & & & \frac{5}{16} & & & & & & & \\
& & & & & \frac{1}{16} & & & & & & & \\
& & & & & & \ddots & & & & & & \\
& & & & & & & \frac{1}{16} & & & & & \\
& & & & & & & \frac{5}{16} & & & & & \\
& & & & & & & \frac{5}{8} & \frac{1}{16} & & & & \\
& & & & & & & \frac{5}{8} & \frac{5}{16} & & & & \\
& & & & & &  & \frac{5}{16} & \frac{5}{8} & \frac{1}{16} & & & \\
& & & & & & & \frac{1}{16} & \frac{5}{8} & \frac{5}{16} & & & \\
& & & & & & & & \frac{5}{16} & \frac{29}{48} & \frac{1}{12} & & \\
& & & & & & & & \frac{1}{16} & \frac{25}{48} & \frac{5}{12} & & \\
& & & & & & & & & \frac{1}{8} & \frac{11}{16} & \frac{3}{16} & \\
& & & & & & & & & & \frac{1}{4} & \frac{3}{4} & \\
& & & & & & & & & & & \frac{1}{2} & \frac{1}{2} \\
& & & & & & & & & & & & 1
\end{array}\right]
\]

\caption{We present the subdivision matrices $A_p$ for $p = 1, 2, 3, 4$. Only the nonzero entries are displayed. Note that, except for the first and last $p$ columns, the columns of $A_p$ are generated by the vector $\boldsymbol{\eta}_p \in \mathbb{R}^{p+2}$, which slides downward with a stride of 2 from one column to the next.}\label{F:knot insertion matrix}
\end{figure}

\section{Left inverses for subdivision matrix for the univariate case}\label{sec:left_inverses}

In spline-based function approximation, it is often necessary to modify the resolution at which splines are represented. Given a (fine) spline space $\SSS$, each spline function $s \in \SSS$ is associated with a coefficient vector ${\bf c}$, which corresponds to its representation in the B-spline basis of $\SSS$. Our objective is to construct an \emph{efficient} method to approximate $s$ in a coarser spline space $\hat{\SSS} \subset \SSS$ by computing the new set of B-spline coefficients ${\bf \hat{c}}$ directly from the original coefficients ${\bf c}$.

As discussed in Section~\ref{sec:b_splines_and_subdivision_matrices}, when refining a spline space, there exists a subdivision matrix $A$ that maps the coarse B-spline coefficients to the fine ones, i.e., $A {\bf \hat{c}} = {\bf c}$. In this context, the problem can be reformulated as follows: find a matrix $B$ such that ${\bf \hat{c}} = B {\bf c}$, for a given ${\bf c}$. In other words, $B$ acts as a \emph{left inverse} of~$A$.

\subsection{An efficient construction of left inverses through a local least-squares method}\label{sec:construction_of_left_inverses}

Several approaches can be employed to obtain a left inverse $B$ of the matrix $A$. 
In this work, we follow the strategy introduced in~\cite{BGS2006-local-least-squares}, where each row of \( B \) is computed by solving a local least-squares problem.
In our setting, the matrix \( A \in \mathbb{R}^{n \times \hat{n}} \) is a subdivision matrix that satisfies~\eqref{knot insertion 1D}. 
The particular structure of subdivision matrices enables a simplified application of the method by Bartels et al., reducing the global least-squares problem to just a few (specifically, three) small local cases, as will be detailed later.

Roughly speaking, the algorithm in~\cite{BGS2006-local-least-squares} defines, for each \( j = 1, \dots, \hat{n} \), the \( j \)-th row of \( B \) as a particular row of the pseudoinverse of a submatrix \( A_j \) of \( A \), namely \( (A_j^T A_j)^{-1} A_j^T \), padded with zeros in the appropriate positions to reach length \( n \).  
Each submatrix \( A_j \) is constructed according to the following procedure.

\begin{procedure}\label{proce:bartels} 
	Select the \( j \)-th column of \( A \), then
	\begin{enumerate}
		\item choose a range of rows from \( A \) that includes at least one nonzero entry in the \( j\)-th column, 
		\item discard any zero columns from the resulting submatrix, and
		\item if necessary, readjust the selected rows and the corresponding nonzero columns so that
		\begin{equation}\label{cond:submatrix-selection} 
			A_j \text{ is a full-rank, square or overdetermined submatrix of } A.
		\end{equation}
	\end{enumerate}
\end{procedure}

As mentioned earlier, by exploiting the banded structure of the subdivision matrix \( A \) and the fact that most of its columns are shifted versions of a vector \( \boldsymbol{\eta} \) (see Figure~\ref{fig:estructura_matriz_A}), it is sufficient to consider only three distinct matrices, namely \( A_{\mathrm{in}} \), \( A_{\mathrm{tl}} \), and \( A_{\mathrm{br}} \), to cover all indices \( j = 1, \dots, \hat{n} \) when defining the submatrices \( A_j \) in the algorithm of Bartels et al.

\begin{figure}[htbp]
  \centering
    \begin{tikzpicture} \label{grafico de A}
        \def\nrow{1.5}       
        \def\nlocal{2.5}     
        \def\ncols{6}        
        \def\nblock{4}       
  
        \draw (0,-1) rectangle (7,8.25);

 		\draw (1,0) rectangle (6,7.25) node[midway,xshift=40pt,yshift=75pt] {$A_{\text{core}}$};
 		 \draw [draw=black, very thick, dashed] (1,0) rectangle (6,7.25);
        \draw (1,6) rectangle (1.5,7.25) node[midway] {\small $\boldsymbol{\eta}$ };
        \draw (1.5,5.25) rectangle (2,6.5) node[midway] {\small $\boldsymbol{\eta}$ };
        \draw (2,4.5) rectangle (2.5,5.75) node[midway] {\small $\boldsymbol{\eta}$ };
        \draw (2.5,3.75) rectangle (3,5) node[midway] {\small $\boldsymbol{\eta}$ };
		\draw (3,3) rectangle (3.5,4.25) node[midway] {\small $\boldsymbol{\eta}$ };
 		\draw (3.5,2.25) rectangle (4,3.5) node[midway] {\small $\boldsymbol{\eta}$ };
         \node at (4.5,2.25) {\Huge $\ddots$};
        \draw (5,0.75) rectangle (5.5,2) node[midway] {\small $\boldsymbol{\eta}$ };
        \draw (5.5,0) rectangle (6,1.25) node[midway] {\small $\boldsymbol{\eta}$ };
        
        \draw [draw=red, very thick, dashed] (0,6.75) rectangle (1.5,8.25) node[midway,xshift=-6pt] {\footnotesize $\textcolor{red}{A_{\text{tl}}}$};
        
        \draw[blue, very thick, dashed] (1.5,3) rectangle (4,5.75) node[midway,xshift=20pt,yshift=20pt] {\footnotesize $\textcolor{blue}{A_{\text{in}}}$};
        
        \draw [draw=green!70!black, very thick, dashed] (5.5,-1) rectangle (7,0.5) node[midway,xshift=8pt] {\footnotesize $\textcolor{green!70!black}{A_{\text{br}}}$};
        
        \draw (1,7.25) rectangle (7,8.25) node[midway] {\small };
        \draw (0,0) rectangle (1,6.75) node[midway] {\small };
        \draw (6,0.5) rectangle (7,7.25) node[midway] {\small };
        \draw (0,-1) rectangle (6,0) node[midway] {\small };
        
        \draw[decorate,decoration={brace,amplitude=5pt,mirror}] (1.5,2) -- (4,2) node[midway,yshift=-12pt] {\small $q$};
         \draw[decorate,decoration={brace,amplitude=5pt,mirror}] (4.2,3) -- (4.2,5.75) node[midway,xshift=12pt] {\small $r$};
        \draw[decorate,decoration={brace,amplitude=5pt,mirror}] (7.2,0) -- (7.2,7.25) node[midway,xshift=25pt] {\small $n-2p$};
        \draw[decorate,decoration={brace,amplitude=5pt,mirror}] (1,-1.2) -- (6,-1.2) node[midway,yshift=-12pt] {\small $\hat{n}-2p$};
        \draw[decorate,decoration={brace,amplitude=5pt}] (-0.2,6.75) -- (-0.2,8.25) node[midway,xshift=-12pt] {\small $t$};
        \draw[decorate,decoration={brace,amplitude=5pt}] (0,8.45) -- (1.5,8.45) node[midway,yshift=12pt] {\small $l$};
        \draw[decorate,decoration={brace,amplitude=5pt,mirror}] (0,-1.2) -- (1,-1.2) node[midway,yshift=-12pt] {\small $p$};
        \draw[decorate,decoration={brace,amplitude=5pt,mirror}] (7.2,7.25) -- (7.2,8.25) node[midway,xshift=12pt] {\small $p$};
        \draw[<->] (2.1,6.5) -- (2.1,7.25) node[midway,xshift=18pt] {\footnotesize $\sigma =2$};
    \end{tikzpicture}
    \caption{Structure of a subdivision matrix $A$ with relevant blocks.}
      \label{fig:estructura_matriz_A}
  \end{figure}
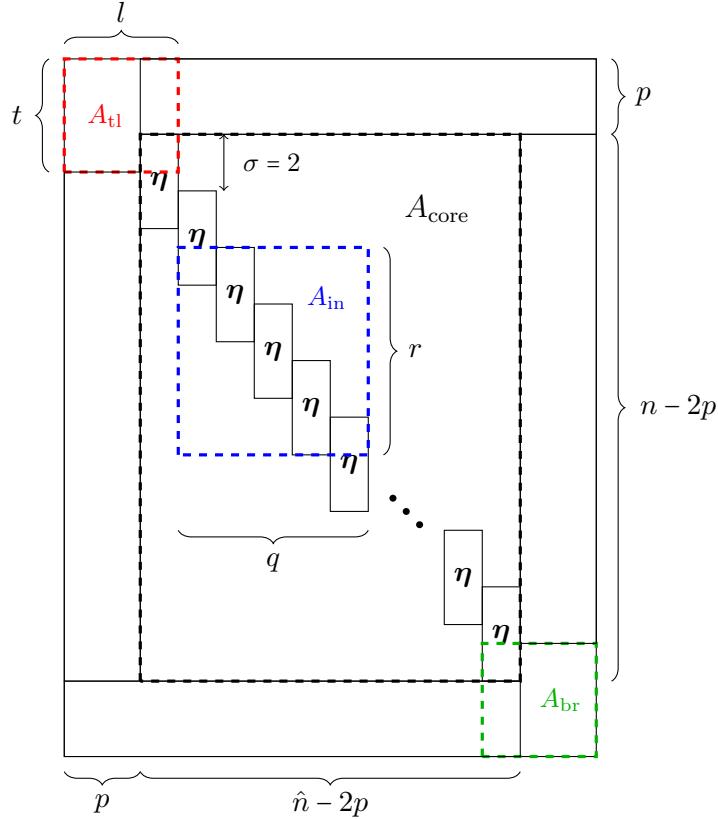

We propose to construct a left inverse \( B \) of the matrix \( A \) as follows. 
For each polynomial degree \( p \), we begin by fixing an index \( k \in \mathbb{N} \) with \( 0 \leq k \leq p+2 \). Based on \( p \) and the chosen value of \( k \), we define the parameter \( r \) by
\begin{equation}\label{eq:r locality width}
r := p+2+2k,
\end{equation}
which plays a central role, and we refer to it as the \emph{locality width}. Particularly, $r$ is the number of entries that are allowed to be nonzero in each interior row of the matrix $B$.

In Figure~\ref{fig:estructura_matriz_A}, the central block of the matrix \( A \), denoted \( A_{\text{core}} \), is a band matrix of size \( (n-2p) \times (\hat{n}-2p) \) with a shift structure. Specifically, each column of \( A_{\text{core}} \) has the same nonzero entries, represented by a vector \( \boldsymbol{\eta} \), but shifted two rows downward (\( \sigma = 2 \)) relative to the preceding column. 
As a starting point, we assign the same matrix \( A_{\mathrm{in}} \) to all interior columns of \( A \) whenever possible. 
Explicitly, for \( \ell+1 \le j \le \hat{n}-\ell \), we set \( A_j = A_{\mathrm{in}} \), where \( A_{\mathrm{in}} \in \mathbb{R}^{r \times q} \) consists of \( r \) rows, with its central column given by \( \boldsymbol{\eta} \) centered at that position (see examples in Figure~\ref{F:matrices locales}). 
The number of columns \( q \) depends on the choice of \( r \) and the adjustments required by Procedure~\ref{proce:bartels} to satisfy condition~\eqref{cond:submatrix-selection}. 
According to~\eqref{eq:r locality width} we have that \( r = p + 2 + 2k \), and then
\begin{equation}\label{eq:q number of columns of Ain}
q := \begin{cases}
	(p+2) + 2\left\lfloor \tfrac{k}{2} \right\rfloor, & \text{if $p$ is odd}, \\
	(p+1) + 2\left\lfloor \tfrac{k+1}{2} \right\rfloor, & \text{if $p$ is even}.
\end{cases}
\end{equation}

\begin{figure}[htbp]
  \centering

  \begin{subfigure}[b]{0.45\textwidth}
      \[
      A_{\mathrm{in}} =
      \begin{bmatrix}
      \frac{3}{4} & \frac{1}{4} &  \\
      \frac{1}{4} & \frac{3}{4} &  \\
       & \frac{3}{4} & \frac{1}{4} \\
       & \frac{1}{4} & \frac{3}{4}
      \end{bmatrix}
      \]
      \caption{Matrix $A_{\mathrm{in}}$ for $p = 2$, locality width 4.}
  \end{subfigure}
  \hfill
  \begin{subfigure}[b]{0.45\textwidth}
      \[
      A_{\mathrm{in}} =
      \begin{bmatrix}
      \frac{1}{4} & \frac{3}{4} &  &  &  \\
       & \frac{3}{4} & \frac{1}{4} &  &  \\
       & \frac{1}{4} & \frac{3}{4} &  &  \\
       &  & \frac{3}{4} & \frac{1}{4} &  \\
       &  & \frac{1}{4} & \frac{3}{4} &  \\
       &  &  & \frac{3}{4} & \frac{1}{4}
      \end{bmatrix}
      \]
      \caption{Matrix $A_{\mathrm{in}}$ for $p = 2$, locality width 6.}
  \end{subfigure}

  \vspace{1em}

  \begin{subfigure}[b]{0.45\textwidth}
      \[
      A_{\mathrm{in}} =
      \begin{bmatrix}
      \frac{1}{8} & \frac{3}{4} & \frac{1}{8} &  &  &  &  \\
       & \frac{1}{2} & \frac{1}{2} &  &  &  &  \\
       & \frac{1}{8} & \frac{3}{4} & \frac{1}{8} &  &  &  \\
       &  & \frac{1}{2} & \frac{1}{2} &  &  &  \\
       &  & \frac{1}{8} & \frac{3}{4} & \frac{1}{8} &  &  \\
       &  &  & \frac{1}{2} & \frac{1}{2} &  &  \\
       &  &  & \frac{1}{8} & \frac{3}{4} & \frac{1}{8} &  \\
       &  &  &  & \frac{1}{2} & \frac{1}{2} &  \\
       &  &  &  & \frac{1}{8} & \frac{3}{4} & \frac{1}{8}
      \end{bmatrix}
      \]
      \caption{Matrix $A_{\mathrm{in}}$ for $p = 3$, locality width 9.}
  \end{subfigure}
  \hfill
  \begin{subfigure}[b]{0.45\textwidth}
      \[
      A_{\mathrm{in}} =
      \begin{bmatrix}
      \frac{1}{2} & \frac{1}{2} &  &  &  &  &  \\
      \frac{1}{8} & \frac{3}{4} & \frac{1}{8} &  &  &  &  \\
       & \frac{1}{2} & \frac{1}{2} &  &  &  &  \\
       & \frac{1}{8} & \frac{3}{4} & \frac{1}{8} &  &  &  \\
       &  & \frac{1}{2} & \frac{1}{2} &  &  &  \\
       &  & \frac{1}{8} & \frac{3}{4} & \frac{1}{8} &  &  \\
       &  &  & \frac{1}{2} & \frac{1}{2} &  &  \\
       &  &  & \frac{1}{8} & \frac{3}{4} & \frac{1}{8} &  \\
       &  &  &  & \frac{1}{2} & \frac{1}{2} &  \\
       &  &  &  & \frac{1}{8} & \frac{3}{4} & \frac{1}{8} \\
       &  &  &  &  & \frac{1}{2} & \frac{1}{2}
      \end{bmatrix}
      \]
      \caption{Matrix $A_{\mathrm{in}}$ for $p = 3$, locality width 11.}
  \end{subfigure}

  \caption{Some examples of submatrices $A_{\text{in}}$ of the subdivision matrix $A$, corresponding to polynomial degrees $p = 2$ and $p = 3$.}\label{F:matrices locales}
\end{figure}

For the remaining columns that cannot be represented by $A_{\mathrm{in}}$, we introduce two additional submatrices of $A$. 
The first one, $A_{\mathrm{tl}} \in \mathbb{R}^{t\times l}$, occupies the top-left corner of $A$ and is used for $1 \le j \le \ell$. 
The second one, $A_{\mathrm{br}} \in \mathbb{R}^{t\times l}$, occupies the bottom-right corner of $A$ and is used for $\hat{n}-\ell+1 \le j \le \hat{n}$. 
The dimensions $t$ and $l$ are chosen so that both $A_{\mathrm{tl}}$ and $A_{\mathrm{br}}$ are consistent with Procedure~\ref{proce:bartels} and satisfy condition~\eqref{cond:submatrix-selection}. The selection of these parameters is constrained but not unique, as will be discussed later (see Remark~\ref{rem:parameters}). An illustration of the three submatrices $A_{\mathrm{tl}}, A_{\mathrm{in}},$ and $A_{\mathrm{br}}$ is given in Figure~\ref{fig:example_submatrices_of_A}.

\begin{figure}[htbp]
    \[
   A_{\mathrm{tl}} =  \begin{bmatrix}
   	   	1 				   & 		 			& 					\\
   	   	\frac{1}{2}  & \frac{1}{2}   & 					 \\
   						   & \frac{3}{4}  & \frac{1}{4}  \\
   		 			 	   & \frac{1}{4}   & \frac{3}{4} 
\end{bmatrix}\quad 
A_{\mathrm{in}} =  \begin{bmatrix}
	\frac{3}{4}  & \frac{1}{4}  &					\\
	\frac{1}{4}	 &  \frac{3}{4} &					 \\
					    & \frac{3}{4} & \frac{1}{4}   \\
						& \frac{1}{4} &  \frac{3}{4}
\end{bmatrix}\quad 
A_{\mathrm{br}} =  \begin{bmatrix}
\frac{3}{4}  & \frac{1}{4} & 				 \\
\frac{1}{4}  & \frac{3}{4} & 			 	  \\
& \frac{1}{2} & \frac{1}{2} \\
&  					&   1
\end{bmatrix}
    \]
    \caption{Main blocks of the matrix $A$ for polynomial degree $p=2$ corresponding to the parameters $r = 4$, $q = 3$, $t = 4$, $l=3$, $\ell = 2$ and $z = 2$. Submatrices $A_{\mathrm{tl}}$,  $A_{\mathrm{in}}$ and $A_{\mathrm{br}}$.}
      \label{fig:example_submatrices_of_A}
  \end{figure}

In summary, the submatrices $A_j$ of $A$ are defined as follows:
\[
A_j := \begin{cases}
	A_{\mathrm{tl}}, & 1 \le j \le \ell, \\
	A_{\mathrm{in}}, & \ell+1 \le j \le \hat{n}-\ell, \\
	A_{\mathrm{br}}, & \hat{n}-\ell+1 \le j \le \hat{n}.
\end{cases}
\]

We now turn to the construction of the matrix $B$. 
The process is carried out block by block, as illustrated in Figure~\ref{fig:estructura_matriz_B}.
The two corner blocks of the matrix \( B \in \mathbb{R}^{\hat{n} \times n} \), denoted by \(B_{\mathrm{tl}}\) and  \(B_{\mathrm{br}}\) in Figure~\ref{fig:estructura_matriz_B}, are obtained from 
the Moore--Penrose pseudoinverses of \( A_{\mathrm{tl}} \) and \( A_{\mathrm{br}} \), namely
\[
\tilde{B}_{\mathrm{tl}} := (A_{\mathrm{tl}}^T A_{\mathrm{tl}})^{-1} A_{\mathrm{tl}}^T, 
\quad 
\tilde{B}_{\mathrm{br}} := (A_{\mathrm{br}}^T A_{\mathrm{br}})^{-1} A_{\mathrm{br}}^T.
\]
We then define \( B_{\mathrm{tl}} \) as the matrix composed of the first \(\ell\) rows of \( \tilde{B}_{\mathrm{tl}} \), whereas \( B_{\mathrm{br}} \) is defined by the last \(\ell\) rows of \( \tilde{B}_{\mathrm{br}} \).

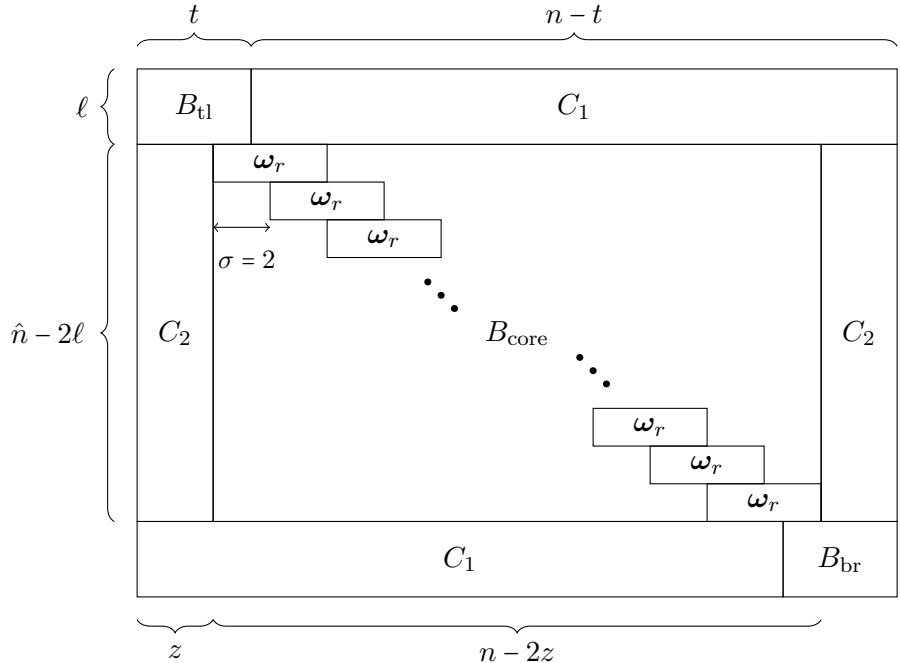
\begin{figure}[htbp]
\centering
  \begin{tikzpicture} 
      \def\nrow{1.5}       
      \def\nlocal{2.5}     
      \def\ncols{6}        
      \def\nblock{4}       


      \draw (1,4) rectangle (2.5,4.5) node[midway] {\small $\boldsymbol{\omega}_{r}$ };
      \draw (1.75,3.5) rectangle (3.25,4) node[midway] {\small $\boldsymbol{\omega}_{r}$ };
      \draw (2.5,3) rectangle (4,3.5) node[midway] {\small $\boldsymbol{\omega}_{r}$ };
      \draw[<->] (1,3.4) -- (1.75,3.4) node[midway,yshift=-12pt, xshift=2pt] {\footnotesize $\sigma =2$};
       \node at (4,2.5) {\Huge $\ddots$};
       \node at (6,1.5) {\Huge $\ddots$};
      \draw (6,0.5) rectangle (7.5,1) node[midway] {\small $\boldsymbol{\omega}_{r}$ };
      \draw (6.75,0) rectangle (8.25,0.5) node[midway] {\small $\boldsymbol{\omega}_{r}$ };
      \draw (7.5,-0.5) rectangle (9,0) node[midway] {\small $\boldsymbol{\omega}_{r}$ };

      \draw (0,4.5) rectangle (1.5,5.5) node[midway] {\small $B_{\text{tl}}$};
      \draw (1.5,4.5) rectangle (10,5.5) node[midway] {\small $C_1$};
      
      \draw (0,-0.5) rectangle (1,4.5) node[midway] {\small $C_2$};
      \draw (9,-0.5) rectangle (10,4.5) node[midway] {\small $C_2$};
      \draw (1,-0.5) rectangle (9,4.5) node[midway] {\small $B_{\text{core}}$};

      \draw (0,-1.5) rectangle (8.5,-0.5) node[midway] {\small $C_1$};
      \draw (8.5,-1.5) rectangle (10,-0.5) node[midway] {\small $B_{\text{br}}$};
      \draw[decorate,decoration={brace,amplitude=5pt}] (-0.3,-0.5) -- (-0.3,4.5) node[midway,xshift=-25pt] {\small $\hat{n}-2\ell$};
      \draw[decorate,decoration={brace,amplitude=5pt}] (-0.3,4.5) -- (-0.3,5.5) node[midway,xshift=-12pt] {\small $\ell$};

      \draw[decorate,decoration={brace,amplitude=5pt}] (0,5.8) -- (1.5,5.8) node[midway,yshift=12pt] {\small $t$};
      \draw[decorate,decoration={brace,amplitude=5pt}] (1.5,5.8) -- (10,5.8) node[midway,yshift=12pt] {\small $n-t$};

      \draw[decorate,decoration={brace,amplitude=5pt,mirror}] (0,-1.8) -- (1,-1.8) node[midway,yshift=-12pt] {\small $z$};
      \draw[decorate,decoration={brace,amplitude=5pt,mirror}] (1,-1.8) -- (9,-1.8) node[midway,yshift=-12pt] {\small $n-2z$};
  \end{tikzpicture}
  \caption{Structure of the proposed left inverse $B$ of the subdivision matrix with relevant blocks.}
      \label{fig:estructura_matriz_B}
\end{figure}

In addition, the central block, denoted $B_{\mathrm{core}}$, is a band matrix of size $(\hat{n}-2\ell) \times (n-2z)$ 
with a shift structure. Each row of $B_{\mathrm{core}}$ contains the same nonzero entries, given by a vector $\boldsymbol{\omega}_{r}$, 
shifted two columns to the right ($\sigma = 2$) relative to the preceding row. 
The vector $\boldsymbol{\omega}_{r}$ is extracted from the central row of
\begin{equation}\label{eq:omega_r_definition}
	B_{\mathrm{in}} = (A_{\mathrm{in}}^T A_{\mathrm{in}})^{-1} A_{\mathrm{in}}^T,
\end{equation}
which is well defined because the number of rows \( q \) is odd (cf.~\eqref{eq:q number of columns of Ain}).

Finally, zero matrices are added to complete the global structure, namely 
$C_1 \in \mathbb{R}^{\ell \times (n-t)}$ and 
$C_2 \in \mathbb{R}^{(\hat{n}-2\ell) \times z}$. It is worth noting that \(z\) represents the number of zero entries that need to be inserted in the \((\ell+1)\)-th row of \(B\) preceding the vector \(\boldsymbol{\omega}_r\). The resulting arrangement is depicted in Figure~\ref{fig:estructura_matriz_B}, 
while specific examples of $B_{\mathrm{tl}}, B_{\mathrm{br}}$, and the vector $\boldsymbol{\omega}_{r}$ 
(the nonzero part of the rows of $B_{\text{core}}$) are shown in Figure~\ref{fig:example_blocks_of_B}.

\begin{figure}[htbp]
 \centering
\[
 B_{\mathrm{tl}} =
 \begin{bmatrix}
 \tfrac{121}{141} & \tfrac{40}{141} & \tfrac{-9}{47} & \tfrac{1}{141} & \tfrac{3}{47} & \tfrac{-1}{47} \\
 \tfrac{-41}{141} & \tfrac{82}{141} & \tfrac{45}{47} & \tfrac{-5}{141} & \tfrac{-15}{141} & \tfrac{5}{47} \\
 \tfrac{5}{47} & \tfrac{-10}{47} & \tfrac{-5}{47} & \tfrac{35}{47} & \tfrac{33}{47} & \tfrac{-11}{47}
 \end{bmatrix}\quad
 B_{\mathrm{br}} =
 \begin{bmatrix}
 \tfrac{-11}{47} & \tfrac{33}{47} & \tfrac{35}{47} & \tfrac{-5}{47} & \tfrac{-10}{47} & \tfrac{5}{47} \\
 \tfrac{5}{47} & \tfrac{-15}{141} & \tfrac{-5}{141} & \tfrac{45}{47} & \tfrac{82}{141} & \tfrac{-41}{141} \\
 \tfrac{-1}{47} & \tfrac{3}{47} & \tfrac{1}{141} & \tfrac{-9}{47} & \tfrac{40}{141} & \tfrac{121}{141}
 \end{bmatrix}\]
\[
\boldsymbol{\omega}_{r} = \boldsymbol{\omega}_{6} =
\begin{bmatrix}
	0  & -\tfrac{1}{4} & \tfrac{3}{4} & \tfrac{3}{4} & -\tfrac{1}{4} & 0  
\end{bmatrix}
\]
 \caption{Submatrices \( B_{\text{tl}},  B_{\text{br}} \) and \(\boldsymbol{\omega}_{r}\) of $B$ for $p=2$ and $r=6$.}
 \label{fig:example_blocks_of_B}
\end{figure}

\begin{remark}[On the parameters involved in the construction of the left inverse]\label{rem:parameters} 
We summarize here the relevant information. Recall that the parameters \(r\) and \(q\) correspond to the size of the matrix \(A_{\mathrm{in}}\in\mathbb{R}^{r\times q}\) and are given in~\eqref{eq:r locality width} and~\eqref{eq:q number of columns of Ain}, respectively. 

On the other hand, \(\ell\) denotes the smallest integer such that the matrix \(A_{\mathrm{in}}\) satisfies the requirements of Procedure~\ref{proce:bartels}, allowing it to be used as the local matrix for all columns of \(A\) except for the first and last \(\ell\) ones.

The parameters \(t\) and \(l\) are chosen so that the matrices \(A_{\mathrm{tl}}\) and \(A_{\mathrm{br}}\), of size \(t\times l\), satisfy the requirements of  Procedure~\ref{proce:bartels}, and can thus be used as the local matrices for the first and last \(\ell\) columns of \(A\), respectively.

Finally, \(z\) specifies the number of zero entries to be inserted in the \((\ell+1)\)-th row of \(B\) before placing the vector \(\boldsymbol{\omega}_r\), in accordance with the assembly procedure of the left inverse.

These requirements do not determine a unique configuration, since multiple parameter choices may yield full-rank submatrices. 
For the purposes of our analysis, we adopt the values reported in Table~\ref{tab:all-p}.
\end{remark}

\begin{table}[h!]
\centering
\scriptsize
\captionsetup[subtable]{justification=centering}
\begin{subtable}{0.48\textwidth}
\centering
\caption*{$p = 1$}
\label{tab:p1}
\begin{tabular}{cccccc}
\toprule
$r$ & $q$ & $t$ & $l$ & $\ell$ & $z$ \\
\midrule
3  & 3 & 2 & 2 & 1 & 1 \\
5  & 3 & 2 & 2 & 1 & 0 \\
7  & 5 & 4 & 3 & 2 & 1 \\
9  & 5 & 4 & 3 & 2 & 0 \\
\bottomrule
\end{tabular}
\end{subtable}\hfill
\begin{subtable}{0.48\textwidth}
\centering
\caption*{$p = 2$}
\label{tab:p2}
\begin{tabular}{cccccc}
\toprule
$r$ & $q$ & $t$ & $l$ & $\ell$ & $z$ \\
\midrule
4  & 3 & 4 & 3 & 2 & 2 \\
6  & 5 & 6 & 4 & 3 & 3 \\
8  & 5 & 6 & 4 & 3 & 2 \\
10 & 7 & 8 & 5 & 4 & 3 \\
12 & 7 & 8 & 5 & 4 & 2 \\
\bottomrule
\end{tabular}
\end{subtable}

\vspace{0.5em}

\begin{subtable}{0.48\textwidth}
\centering
\caption*{$p = 3$}
\label{tab:p3}
\begin{tabular}{cccccc}
\toprule
$r$ & $q$ & $t$ & $l$ & $\ell$ & $z$ \\
\midrule
5  & 5 & 8  & 6 & 4 & 5 \\
7  & 5 & 8  & 6 & 4 & 4 \\
9  & 7 & 10 & 7 & 5 & 5 \\
11 & 7 & 10 & 7 & 5 & 4 \\
13 & 9 & 12 & 8 & 6 & 5 \\
15 & 9 & 12 & 8 & 6 & 4 \\
\bottomrule
\end{tabular}
\end{subtable}\hfill
\begin{subtable}{0.48\textwidth}
\centering
\caption*{$p = 4$}
\label{tab:p4}
\begin{tabular}{cccccc}
\toprule
$r$ & $q$ & $t$ & $l$ & $\ell$ & $z$ \\
\midrule
6  & 5  & 10 & 7  & 5 & 6 \\
8  & 7  & 12 & 8  & 6 & 7 \\
10 & 7  & 12 & 8  & 6 & 6 \\
12 & 9  & 14 & 9  & 7 & 7 \\
14 & 9  & 14 & 9  & 7 & 6 \\
16 & 11 & 16 & 10 & 8 & 7 \\
18 & 11 & 16 & 10 & 8 & 6 \\
\bottomrule
\end{tabular}
\end{subtable}
\caption{
The parameters correspond to the local matrix 
$A_{\mathrm{in}} \in \mathbb{R}^{p \times q}$, 
which is used for all interior columns of the subdivision matrix $A$, 
except for the first and last $\ell$ columns, 
where the local matrices 
$A_{\mathrm{tl}}$ and $A_{\mathrm{br}}$ of size $t \times l$ 
are employed, respectively.
The value of the parameter $z$, used in the assembly of the left inverse 
matrix $B$, is also reported.
}
\label{tab:all-p}
\end{table}

The above construction yields the following result.

\begin{proposition}\rm
	Let \( A \in \mathbb{R}^{n \times \hat{n}} \) be a subdivision matrix defined by~\eqref{knot insertion 1D}, 
	and let \( B \in \mathbb{R}^{\hat{n} \times n} \) be the matrix constructed as above. Then,
	\[
	BA = I_{\hat{n}}.
	\]
\end{proposition}

\begin{proof}
	Since the submatrices \( A_{\mathrm{tl}}, A_{\mathrm{br}}, A_{\mathrm{in}} \) were selected according to Procedure~\ref{proce:bartels} 
	and satisfy condition~\eqref{cond:submatrix-selection}, 
	they also fulfill the hypotheses of~\cite[Theorem~2.1]{BGS2006-local-least-squares}. 
	Consequently, the assembled matrix \( B \) is a left inverse of \( A \), that is, \( BA = I_{\hat{n}} \).
\end{proof}

\begin{remark}\label{rem:realition_with_Golub_&_Samavati_2006}
	As suggested in~\cite{BGS2006-local-least-squares}, 
	a natural strategy for constructing a left inverse is to build it row by row, 
	ensuring that each row satisfies the identity condition when multiplied with \( A \), i.e., \( BA = I_{\hat{n}} \). 
	Specifically, in the case of constructing the \( j \)-th row of a left inverse matrix \( B \) for $\ell < j \leq \hat{n}-\ell$,  
	we seek a row vector \( \mathbf{x}^T \in \mathbb{R}^{1 \times r} \) satisfying 
	\[
	\mathbf{x}^T A_{\mathrm{in}} = \mathbf{e}_i^T,
	\]
	which, upon transposition, leads to the linear system
	\begin{equation}\label{eq:system_for_omega_r}
		A_{\mathrm{in}}^T \mathbf{x} = \mathbf{e}_i.
	\end{equation}
	Here, \( \mathbf{e}_i \) denotes the \( i \)-th canonical basis vector of \( \mathbb{R}^q \), with \( i \) typically taken to be the central index, which is well defined because \( q \) is odd. 
	Notice that~\eqref{eq:system_for_omega_r} in general admits infinitely many solutions since \( A_{\mathrm{in}} \) has full column rank (by condition~\eqref{cond:submatrix-selection}) and \( q \le r \). 
	Choosing $\boldsymbol{\omega}_r$ as the central row of~\eqref{eq:omega_r_definition} amounts to selecting, among all possible solutions of~\eqref{eq:system_for_omega_r}, the least-squares solution, i.e.,
	\[
	\boldsymbol{\omega}_r^T = A_{\mathrm{in}} (A_{\mathrm{in}}^T A_{\mathrm{in}})^{-1} \mathbf{e}_i.
	\]
	Thus, by standard results in least-squares theory, $\boldsymbol{\omega}_r^T$ is the \emph{minimum-norm solution} to the system~\eqref{eq:system_for_omega_r} (see, e.g.,~\cite{GOLUB2013MC}).
\end{remark}

The next section presents arguments that justify selecting the minimum-norm solution from the set of all possible solutions of~\eqref{eq:system_for_omega_r}.

\subsection{On the importance of choosing the minimum-norm solution}

As we already mention, the matrix \( B \in \mathbb{R}^{\hat{n} \times n} \) is a block matrix. The central and main submatrix of \(B\), \(B_{\text{core}}\), is sparse. 
Each row contains a block \( \boldsymbol{\omega}_r \in \mathbb{R}^r \) located at different positions, with zeros elsewhere. Thus, it holds that \(\|B_{\text{core}}\|_\infty=\|\boldsymbol{\omega}_r\|_1\). 
This property allows us to bound \( \|B\|_\infty \) in terms of \( \|\boldsymbol{\omega}_r\|_2 \). 
Indeed, using norm equivalence in \( \mathbb{R}^n \), we have that
\begin{align*}
\|\boldsymbol{\omega}_r\|_2 \le  \|\boldsymbol{\omega}_r\|_1 = \|B_{\text{core}}\|_\infty \le \|B\|_{\infty}. 
\end{align*}
On the other hand, since $\boldsymbol{\omega}_r$ has at most \(r\) non-zero entries, $\|B_{\text{core}}\|_\infty = \|\boldsymbol{\omega}_r\|_1 \le  \sqrt{r} \|\boldsymbol{\omega}_r\|_2,$ and so
\begin{align}\label{eq:upper_bound_B_infty_norm}
\|B\|_\infty \le \max\{\|B_{\text{tl}}\|_\infty,\|B_{\text{br}}\|_\infty,\|B_{\text{core}}\|_\infty\}\le \max\{\|B_{\text{tl}}\|_\infty,\|B_{\text{br}}\|_\infty,\sqrt{r} \|\boldsymbol{\omega}_r\|_2\}.
\end{align}

These two estimates reveal that controlling \( \|\boldsymbol{\omega}_r\|_2 \) plays a central role in bounding \( \|B\|_\infty \). From a numerical standpoint, having a bound on \( \|B\|_\infty \) is advantageous, as it implies a form of stability when applying the inverse operator. Specifically, the inequality
\[
\|B \mathbf{c}\|_\infty \leq \|B\|_\infty \| \mathbf{c} \|_\infty
,\]
shows that large values of \( \|B\|_\infty \) may amplify errors or noise in the vector \( \mathbf{c} \). Therefore, it seems that a \(\boldsymbol{\omega}_r\) with small Euclidean norm directly contributes to the numerical robustness of the inverse process.

On the other hand, suppose we consider two locality widths \( r < r' \), with corresponding \(A_{\text{in}}\) submatrices \( A_r \in \mathbb{R}^{r \times q} \) and \( A_{r'} \in \mathbb{R}^{r' \times q'} \). Then the solution set of the system
\[
A_r^T \mathbf{x} = \mathbf{e}_i,
\]
is contained in the solution set of the larger system
\[
A_{r'}^T \mathbf{y} = \mathbf{e}_{i'}.
\]
This inclusion holds because the smaller system corresponds to a restriction of the larger one, and any solution for \( A_r \) can be extended (e.g., by padding with zeros) to a solution for \( A_{r'} \).
In particular, let \( \boldsymbol{\omega}_r \) and \( \boldsymbol{\omega}_{r'} \) denote the unique minimum-norm solutions of the systems. Then, due to the nestedness of the solution spaces, it follows that
\[
\| \boldsymbol{\omega}_{r'} \|_2 \leq \| \boldsymbol{\omega}_r \|_2.
\]
This inequality reflects the fact that the minimum norm over a larger feasible set cannot exceed that over a smaller one.

Summarizing, in terms of this formal analysis, increasing locality width \( r \) enlarge the size of \( \boldsymbol{\omega}_r \), allowing for smaller Euclidean norms, but this benefit can be partially offset by the growth of the factor \( \sqrt{r} \)  in the upper bound~\eqref{eq:upper_bound_B_infty_norm}. 
Thus, there is a trade-off between locality, sparsity of the inverse and numerical stability, where choosing an appropriate \( r \) is essential to control both the norm of \( B \) and the computational cost of applying it. Some explicit numerical computations on the \( \|B\|_{\infty} \) and \( \|\boldsymbol{\omega}_r\|_{2}\) for different values of $p$ and $r$ are presented in Table~\ref{tab:coarsening1D}.

\paragraph{Some explicit expressions for  $\boldsymbol{\omega}_r$.} We now report the vectors \( \boldsymbol{\omega}_r \), 
which characterize the left inverses associated with different locality widths.  
Each vector \(\boldsymbol{\omega}_r\) can be written as \(\boldsymbol{\omega}_r = \alpha_r \boldsymbol{\mu}_r\), 
with the property that $\alpha_r^{-1}$ and the components of $\boldsymbol{\mu}_r$ are integers. In Tables~\ref{tab:omega_for_linear_bsplines},~\ref{tab:omega_for_quadratic_bsplines},~\ref{tab:omega_for_cubic_bsplines} and~\ref{tab:omega_for_quartic_bsplines}, the components of \( \boldsymbol{\omega}_r \) are shown for increasing values of \( r \), displaying the values of 
\(\alpha_r\) and \(\boldsymbol{\mu}_r\), corresponding to the cases \(p=1,2,3,4\), respectively.

\begin{table}[h!]
  \centering \scriptsize
  \begin{tabular}{|c|c|c c c c c c c c c|}
  \hline
  $r$ & $\alpha_r$ & \multicolumn{9}{c|}{$\boldsymbol{\mu}_r$}  \\
  \hline
  $3$ & $1$ &  &  &  & 0 & 1 & 0 &  &  & \\
  \hline
  $5$ & $\frac{1}{7}$ &  &  & -1 & 2 & 5 & 2 & -1 & & \\
  \hline
  $7$ & $\frac{1}{7}$ &  & 0 & -1 & 2 & 5 & 2 & -1 & 0 &  \\
  \hline
  $9$ & $\frac{1}{41}$ & 1 & -2 & -5 & 12 & 29 & 12 & -5 & -2 & 1 \\
  \hline
  \end{tabular}
  \caption{Vectors $\boldsymbol{\omega}_r = \alpha_r \boldsymbol{\mu}_r$ for linear B-splines ($p=1$) for different values of \( r \).}
  \label{tab:omega_for_linear_bsplines}
\end{table}

\begin{table}[h] 
    \centering \scriptsize
\begin{tabular}{|c|c|c c c c c c c c c c c c|}
\hline
$r$& $\alpha_r$ &  & &  &  &  & $\boldsymbol{\mu}_r$ & &  &  &  &  & \\
\hline
$4$ & $\frac{1}{4}$  &  &  &  &  & -1 & 3 & 3 & -1 &  &  &  & \\
\hline
$6$ & $\frac{1}{4}$   &  &  &  & 0 & -1 & 3 & 3 & -1 & 0 &  &  & \\
\hline
$8$ & $\frac{1}{40}$ &  &  & 3 & -9 & -1 & 27 & 27 & -1 & -9 & 3 &  & \\
\hline
$10$ & $\frac{1}{40}$   &  & 0 & 3 & -9 & -1 & 27 & 27 & -1 & -9 & 3 & 0 & \\
\hline
$12$ & $\frac{1}{364}$ & -9 & 27 & 3 & -81 & -1 & 243 & 243 & -1 & -81 & 3 & 27 & -9 \\
\hline
\end{tabular}
\caption{Vectors \(\boldsymbol{\omega}_r = \alpha_r \boldsymbol{\mu}_r \) for quadratic B-splines ($p=2$) for different values of \( r \).}
\label{tab:omega_for_quadratic_bsplines}
\end{table}

\begin{table}[h!]
  \centering \scriptsize
  \begin{tabular}{|c|c|c c c c c c c c|}
  \hline
  $r$ & $\alpha_r$ &  &  &  & $\boldsymbol{\mu}_r$ &  &  &  &  \\
  \hline
  $5$ & $\frac{1}{4}$ &  &  &  &  &  & 0 & -2 & 8 \\
  \hline
  $7$ & $\frac{1}{196}$  &  & & &  & 23 & -92 & 63 & 208  \\
  \hline
  $9$ & $\frac{1}{196}$ & & & & 0 & 23 & -92 & 63 & 208 \\
  \hline
  $11$ & $\frac{1}{12038}$ & & & -569 & 2276 & -1833 & -4048 & 4479 & 11428\\
  \hline
  $13$ & $\frac{1}{12038}$ & & 0 & -569 & 2276 & -1833 & -4048 & 4479 & 11428 \\
  \hline
  $15$ & $\frac{1}{692104}$ & 14351 & -57404 & 46919 &  99344 & -128105 &-213916 & 263423 & 644480 \\
  \hline
  \end{tabular}
  \caption{Portion of vectors \(\boldsymbol{\omega}_r = \alpha_r \boldsymbol{\mu}_r \) for cubic B-splines ($p=3$) for different values of \( r \). The remaining values are obtained by symmetry.}
   \label{tab:omega_for_cubic_bsplines}
\end{table}
\begin{table}[h!]
  \centering \scriptsize
  \begin{tabular}{|c|c|c c c c c c c c c|}
  \hline
  $r$ & $\alpha_r$ &  &  &  & $\boldsymbol{\mu}_r$ &  &  &  &  &  \\
  \hline
  $6$  & $\frac{1}{16}$ &  &  &  &  &  &  & 3 & -15 & 20 \\
  \hline
  $8$  & $\frac{1}{16}$ &  &  &  &  &  & 0 & 3 & -15 & 20  \\ 
  \hline
  $10$ & $\frac{1}{1936}$ &  &  &  &  & -130 & 650 & -937 & -515 & 1900 \\ 
  \hline
  $12$ & $\frac{1}{1936}$ &  &  &  & 0 & -130 & 650 & -937 & -515 & 1900 \\ 
  \hline
  $14$ & $\frac{1}{20704}$ &  &  & 665 & -3325 & 4930 & 1950 & -9993 & -2875 & 19000 \\ 
  \hline
  $16$ & $\frac{1}{20704}$  &  & 0 & 665 & -3325 & 4930 & 1950 & -9993 & -2875 & 19000 \\ 
  \hline
  \end{tabular}
  \caption{Portion of vectors $\boldsymbol{\omega}_r = \alpha_r \boldsymbol{\mu}_r$ for quartic B-splines ($p=4$) for different values of~\( r \). Remaining entries are obtained by symmetry.}
  \label{tab:omega_for_quartic_bsplines}
\end{table}

\subsection{B-spline ancestors and local coarsening}

We begin by showing that if a spline has support strictly contained within the interval \([a,b]\), then, unlike the \(L^2\)-projection and other global methods that yield a coarse approximation with support covering the entire interval, the coarsening operators proposed in this article act locally. In particular, the support of the resulting approximation extends only slightly beyond that of the original spline, with the precise enlargement depending on the locality width used to construct the operator.

More precisely, just as each coarse B-spline has children in the fine space, 
we can associate ancestors in the coarse space to each fine B-spline, as described below.

The children of a coarse B-spline are the fine B-splines required to express it as a linear combination (see the bottom part of Figure~\ref{F:two_scale}). It is well known that the indices of the children of the \(j\)-th B-spline correspond to the row indices with nonzero entries in the \(j\)-th column of the subdivision matrix \(A\), as discussed in Section~\ref{sec:b_splines_and_subdivision_matrices}.

Each left inverse \(B\) of the subdivision matrix \(A\) induces a notion of ancestry. In this context, the ancestors of a fine B-spline are the coarse B-splines that the coarsening operator effectively uses to construct its approximation. Therefore, the number of ancestors depends on the locality width $r$ used to build \(B\), and corresponds to the number of nonzero entries in each column of \(B\). 

Specifically, if we consider fine B-splines whose support is sufficiently internal to the interval $[a,b]$, the number of ancestors depends only on the number of nonzero components of the vector $\boldsymbol{\omega}_r$, denoted by $\|\boldsymbol{\omega}_r\|_0$. If $\|\boldsymbol{\omega}_r\|_0$ is odd, the number of ancestors is $(\|\boldsymbol{\omega}_r\|_0-1)/2$ for some B-splines and $(\|\boldsymbol{\omega}_r\|_0+1)/2$ for others; whereas if $\|\boldsymbol{\omega}_r\|_0$ is even, the number of ancestors is always $\|\boldsymbol{\omega}_r\|_0/2$. 

This fact is illustrated in Figure~\ref{fig:ancestors}, which shows a central submatrix of $B$ with the nonzero entries marked by squares. The number of nonzero entries in the columns corresponding to the black squares confirms the validity of the above formula for the number of ancestors.

\begin{figure}[htbp]
    \centering
    \begin{subfigure}[b]{0.45\textwidth}
        \centering
        \includegraphics[width=\textwidth]{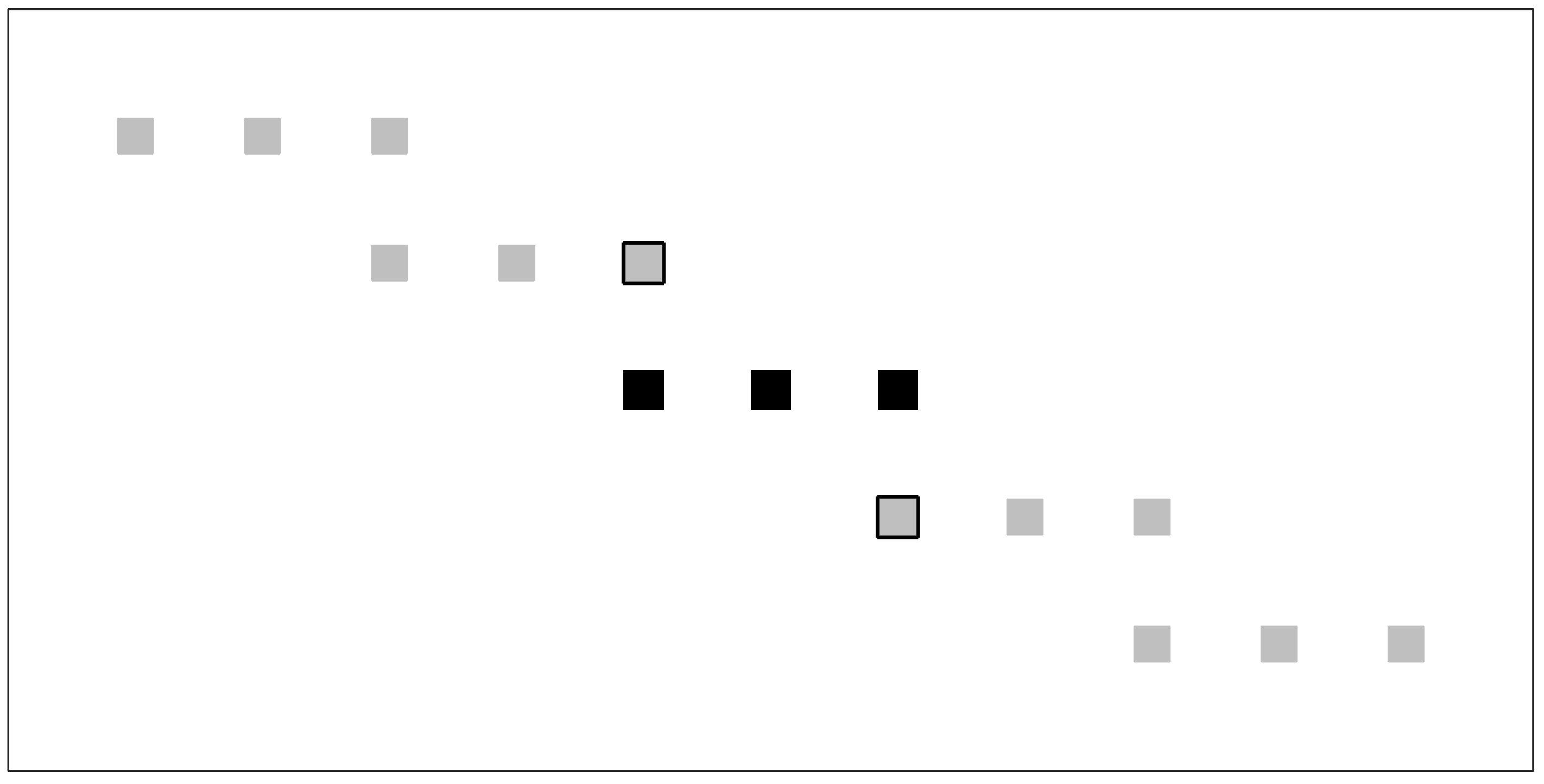}
        \caption{$\|\omega_r\|_0 = 3$}
    \end{subfigure}
    \hfill
    \begin{subfigure}[b]{0.45\textwidth}
        \centering
        \includegraphics[width=\textwidth]{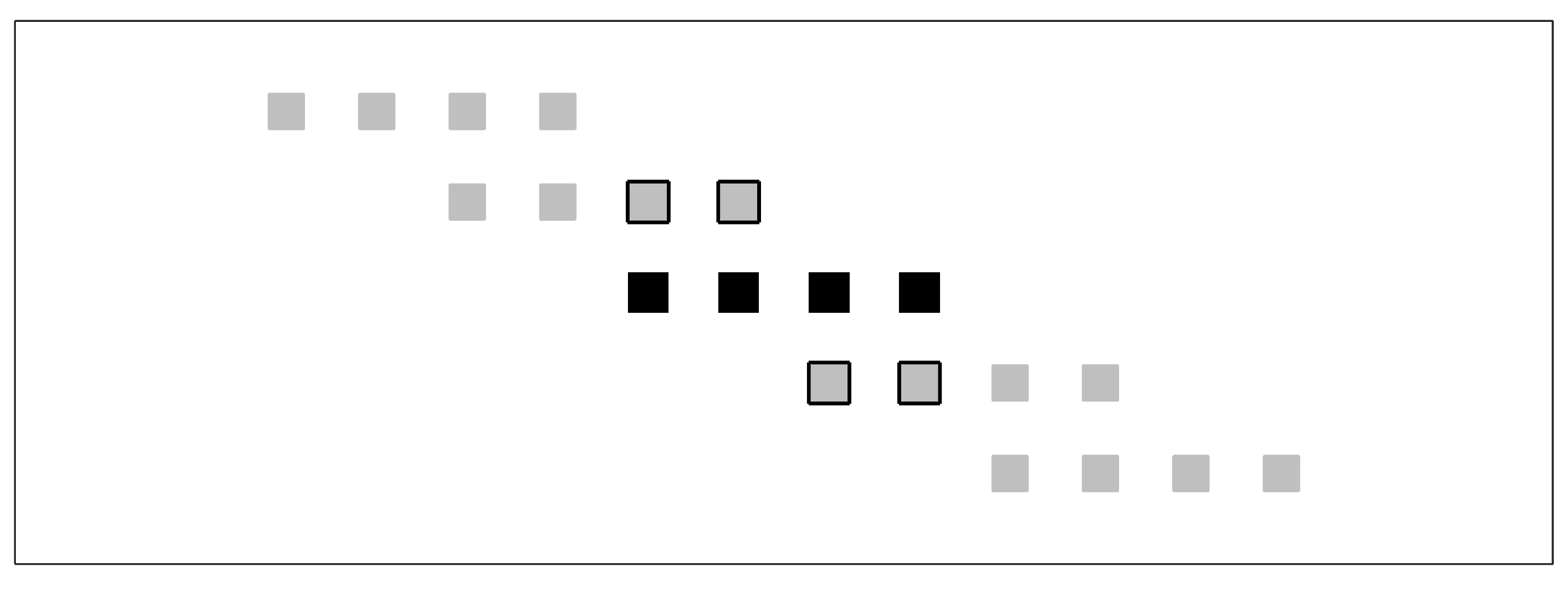}
        \caption{$\|\omega_r\|_0 = 4$}
    \end{subfigure}
    
    \vspace{0cm} 
    
    \begin{subfigure}[b]{0.45\textwidth}
        \centering
        \includegraphics[width=\textwidth]{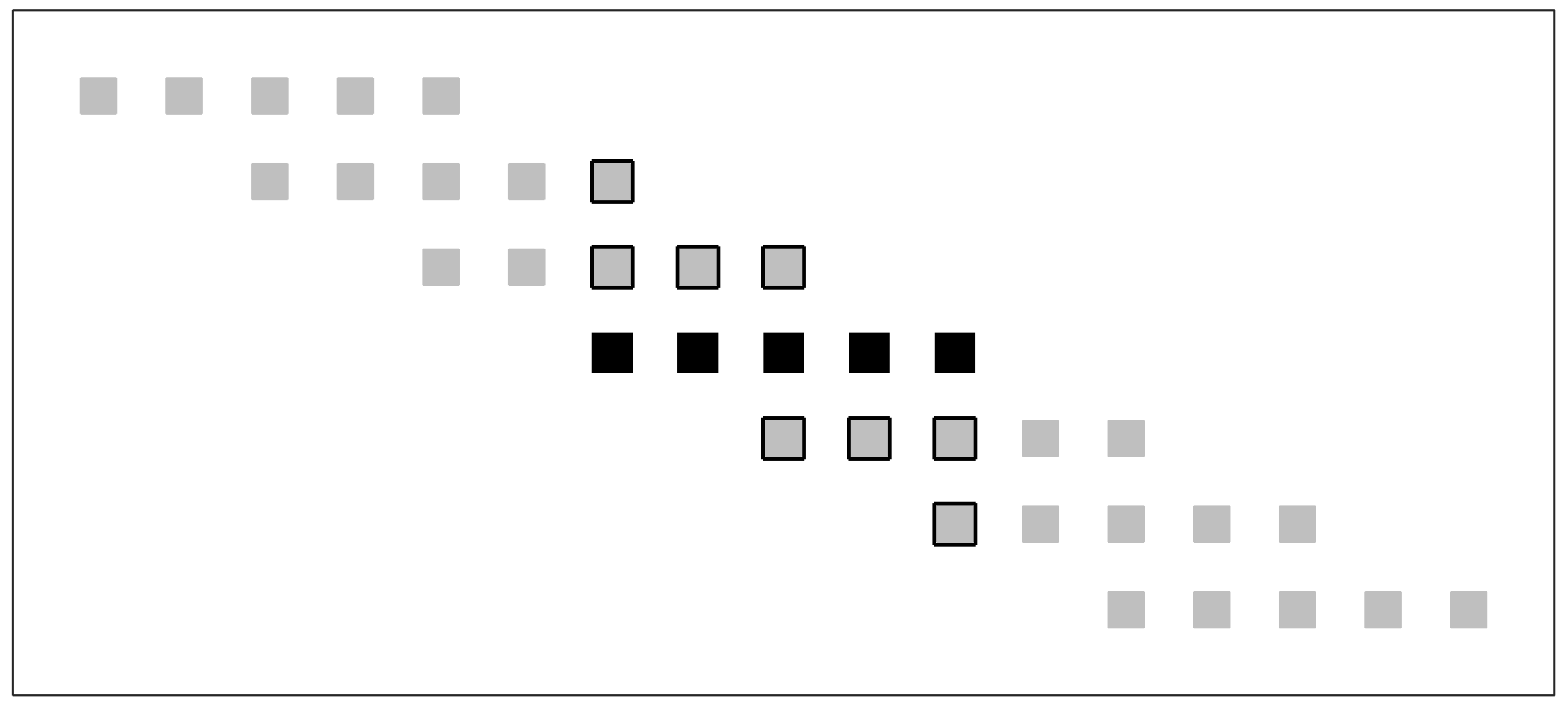}
        \caption{$\|\omega_r\|_0 = 5$}
    \end{subfigure}
    \hfill
    \begin{subfigure}[b]{0.45\textwidth}
        \centering
        \includegraphics[width=\textwidth]{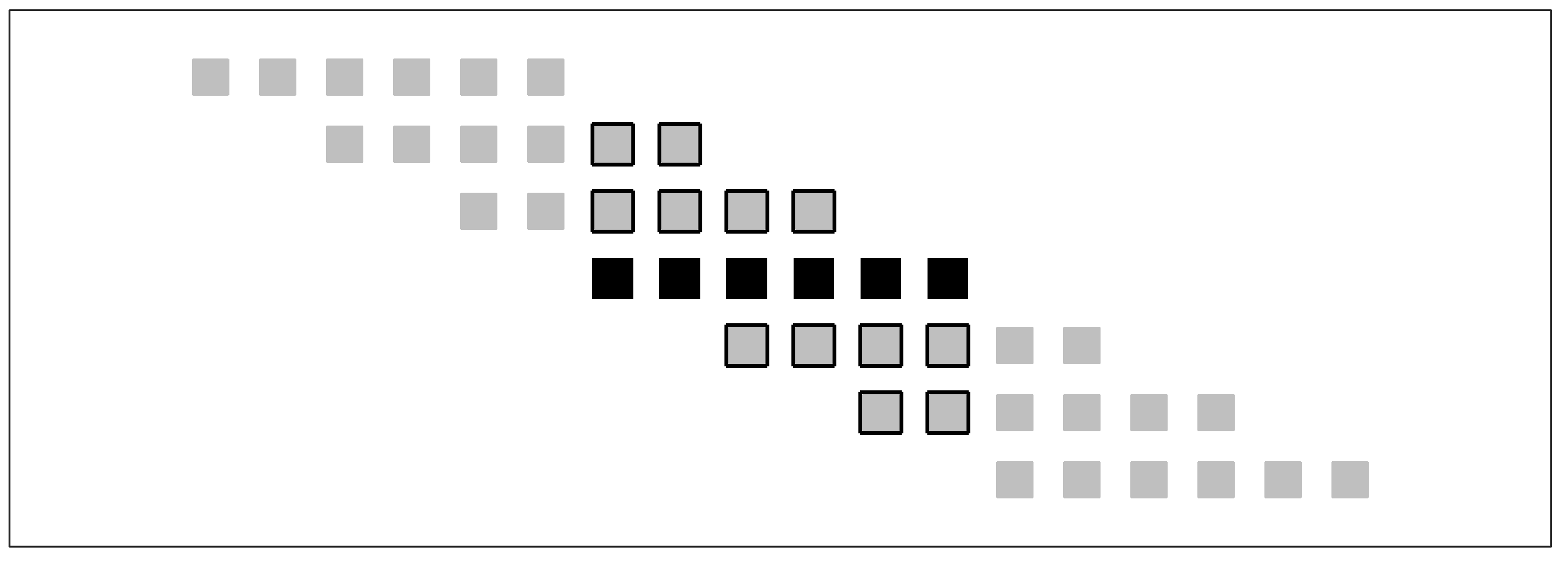}
        \caption{$\|\omega_r\|_0 = 6$}
    \end{subfigure}
    
    \vspace{0cm} 
    
    \begin{subfigure}[b]{0.45\textwidth}
        \centering
        \includegraphics[width=\textwidth]{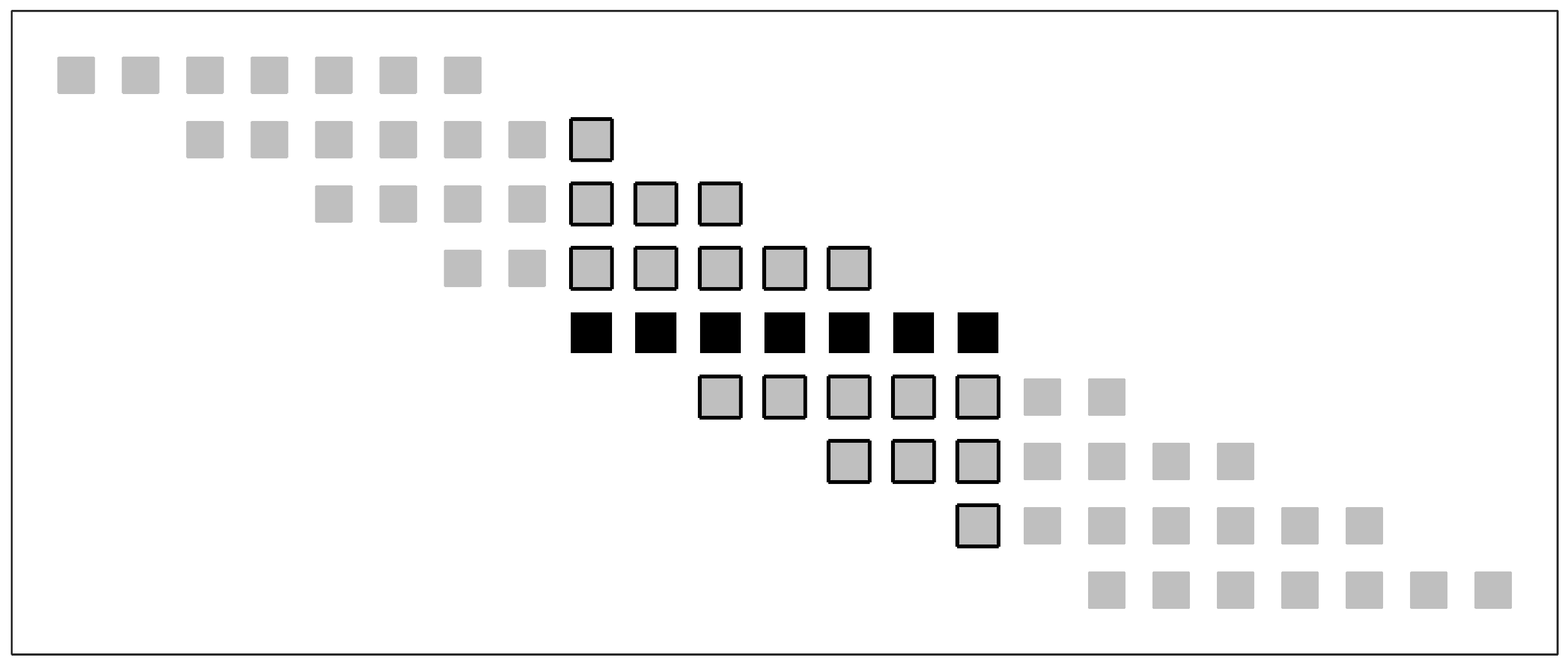}
        \caption{$\|\omega_r\|_0 = 7$}
    \end{subfigure}
    \hfill
    \begin{subfigure}[b]{0.45\textwidth}
        \centering
        \includegraphics[width=\textwidth]{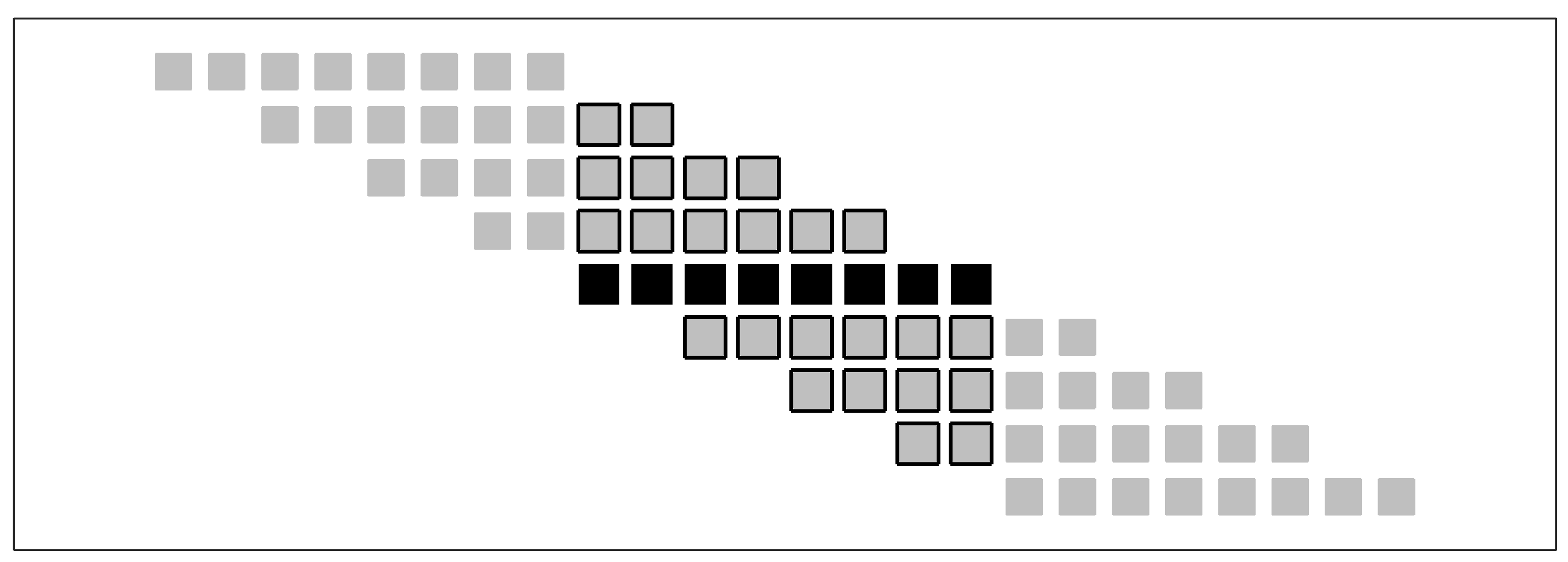}
        \caption{$\|\omega_r\|_0 = 8$}
    \end{subfigure}

    \caption{Schematic representation of the submatrix $B_{\textrm{core}}$ of the operator $B$ for different locality widths $r$. Each square indicates a nonzero entry; black squares mark the central vector. The nonzero entries in each highlighted column reveals the set of ancestors associated with each fine B-spline.}
    \label{fig:ancestors}
\end{figure}

In Figure~\ref{fig:ancestors_deg2}, we display the ancestors of quadratic B-splines for locality widths $r=4$ (middle) and $r=8$ (right). In this case, all B-splines have exactly $r/2$ ancestors. Moreover, two consecutive B-splines share the same set of ancestors. Both observations are consistent with the behavior shown in Figure~\ref{fig:ancestors} (right).

\begin{figure}[htbp]
	\centering
	\includegraphics[width=0.95\linewidth]{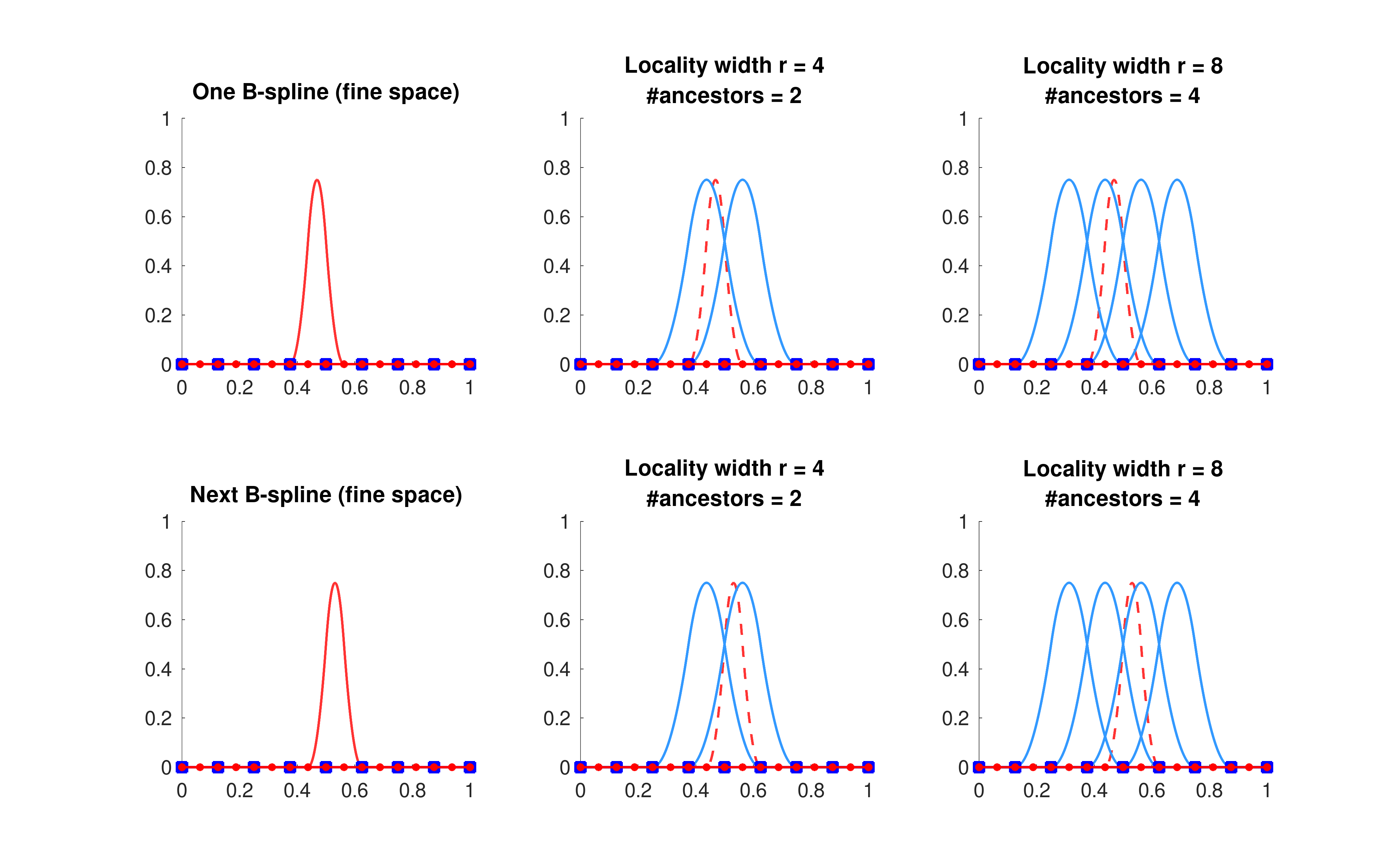}
	\caption{Ancestors of fine quadratic B-splines ($p=2$) in the coarse space, for different locality widths.}
	\label{fig:ancestors_deg2}
\end{figure}

We next turn to cubic B-splines, where the situation differs. The case of cubic B-splines is shown in Figure~\ref{fig:ancestors_deg3}, with locality widths $r=5$ (middle) and $r=7$ (right). In contrast to the quadratic case, for a fixed locality width the number of ancestors depends on whether the support of the fine B-spline coincides with a union of complete coarse intervals or, instead, whether its boundary knots belong exclusively to the fine mesh. Once again, this behavior is consistent with the structure of the matrix $B$ in Figure~\ref{fig:ancestors} (left).

\begin{figure}[htbp]
	\centering
	\includegraphics[width=0.95\linewidth]{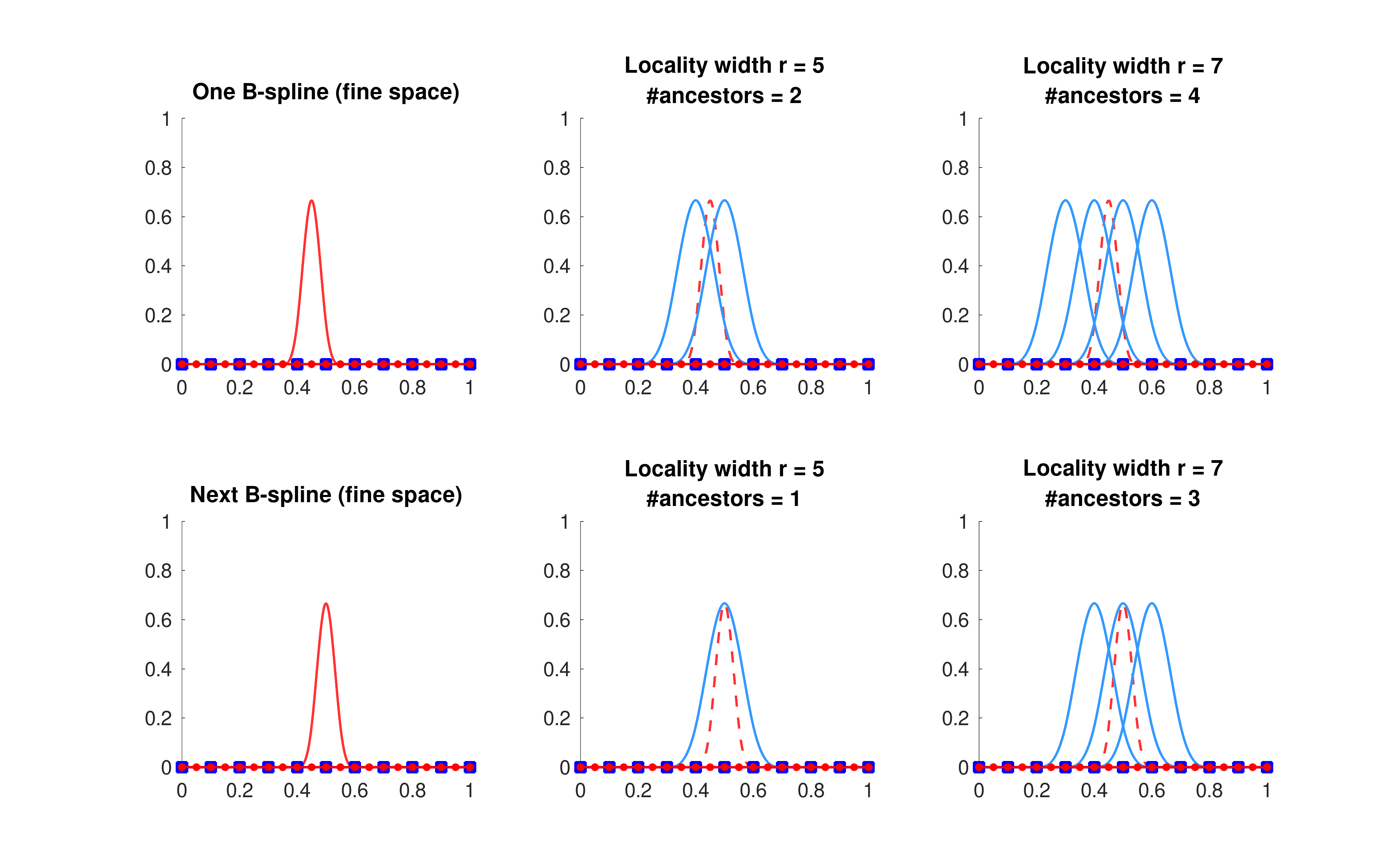}
	\caption{Ancestors of fine cubic B-splines ($p=3$) in the coarse space, for different locality widths.}
	\label{fig:ancestors_deg3}
\end{figure}

Finally, it should be noted that a smaller locality width $r$ gives rise to a more localized coarsening operator, introducing modifications solely within a neighborhood in which the original spline cannot be represented exactly in the coarse space.

\section{Coarsening operators for tensor-product splines}\label{sec:tensor product}

Let $\hat{\Xi}_x$ and $\hat{\Xi}_y$ be $(p+1)$-open knot vectors in the $x$- and $y$-directions, respectively, each associated with a uniform partition of given intervals in $\mathbb{R}$.  
Let $\Xi_x$ and $\Xi_y$ denote the refined knot vectors obtained by inserting the midpoints of all subintervals in $\hat{\Xi}_x$ and $\hat{\Xi}_y$, as decribed in Section~\ref{sec:b_splines_and_subdivision_matrices}.  

We denote by $\hat{\SSS}$ and $\SSS$ the corresponding tensor-product spline spaces spanned by the B-splines defined over the knot vectors
\[ \hat{\Xi} = \hat{\Xi}_x \times \hat{\Xi}_y,
\quad
\Xi = \Xi_x \times \Xi_y, \]
respectively. More especifically,
\[ \hat{\BB} = \{\hat{\beta}_1, \ldots, \hat{\beta}_{\hat{n}}\} 
= \{\hat{\beta}_i^x \hat{\beta}_j^y \}_{i=1,\dots,\hat{n}_x}^{j=1,\dots,\hat{n}_y},
\quad
\BB = \{\beta_1, \ldots, \beta_n\} 
= \{\beta_i^x \beta_j^y \}_{i=1,\dots,n_x}^{j=1,\dots,n_y},\]
denote the tensor-product B-spline bases for $\hat{\SSS}$ and $\SSS$.  
Each bivariate basis function $\hat{\beta}_i$ (resp.\ $\beta_i$) is a product of univariate B-splines in the $x$- and $y$-directions, with superscripts indicating the variable when needed.  
For instance, $\hat{\beta}^x_i(x)$ denotes a univariate coarse B-spline in the $x$-direction.

Following the notation introduced in Section~\ref{sec:b_splines_and_subdivision_matrices}, we consider the vector functions that collect the univariate B-splines in the $x$- and $y$-directions, both in the fine and the coarse space, that is,
\[ \boldsymbol{\beta}^x(x) := 
\begin{bmatrix}
	\beta^x_1(x) \\
	\beta^x_2(x) \\
	\vdots \\
	\beta^x_{n_x}(x)
\end{bmatrix},
\qquad
\boldsymbol{\beta}^y(y) := 
\begin{bmatrix}
	\beta^y_1(y) \\
	\beta^y_2(y) \\
	\vdots \\
	\beta^y_{n_y}(y)
\end{bmatrix},
\qquad
\hat{\boldsymbol{\beta}}^x(x) := 
\begin{bmatrix}
	\hat{\beta}^x_1(x) \\
	\hat{\beta}^x_2(x) \\
	\vdots \\
	\hat{\beta}^x_{\hat{n}_x}(x)
\end{bmatrix},
\qquad
\hat{\boldsymbol{\beta}}^y(y) := 
\begin{bmatrix}
	\hat{\beta}^y_1(y) \\
	\hat{\beta}^y_2(y) \\
	\vdots \\
	\hat{\beta}^y_{\hat{n}_y}(y)
\end{bmatrix}
.\]

Now, if $A_x \in \mathbb{R}^{n_x \times \hat{n}_x}$ and $A_y \in \mathbb{R}^{n_y \times \hat{n}_y}$ denote
the univariate subdivision matrices in the $x$- and $y$-direction, respectively, we have that
\begin{equation*}\label{E:subdivision matrix basis functions 2} \hat{\boldsymbol{\beta}}^x(x) = A_x^{T} \, \boldsymbol{\beta}^x(x), \qquad \hat{\boldsymbol{\beta}}^y(y) = A_y^{T} \, \boldsymbol{\beta}^y(y).\end{equation*}
Taking into account the \emph{vectorization operator}, denoted by $\vect(\cdot)$, that stacks the columns of a matrix into a single column vector, and the \emph{property of the Kronecker product for matrix equations} which establishes that
\[\vect(NXM)=(M^T\otimes N)\vect(X),\]
for arbitrary matrices $M$, $N$ and $X$, we obtain that
\begin{equation}\label{eq:vec knot insertion}\vect(\hat{\boldsymbol{\beta}}^x(x)\hat{\boldsymbol{\beta}}^y(y)^T) = \vect(A_x^{T} \, \boldsymbol{\beta}^x(x) \, \boldsymbol{\beta}^y(y)^TA_y)= (A_y^T\otimes A_x^T)\vect(\boldsymbol{\beta}^x(x) \, \boldsymbol{\beta}^y(y)^T).
\end{equation}

Let $\boldsymbol{\beta}(x,y)$ and $\hat{\boldsymbol{\beta}}(x,y)$ be the vector functions of fine and coarse tensor-product B-splines arranged so that the $i$-index varies fastest, which means that
\[\hat{\boldsymbol{\beta}}(x,y) := \vect(\hat{\boldsymbol{\beta}}^x(x)\hat{\boldsymbol{\beta}}^y(y)^T),\qquad  {\boldsymbol{\beta}}(x,y) := \vect({\boldsymbol{\beta}}^x(x){\boldsymbol{\beta}}^y(y)^T).\]
Hence, in view of~\eqref{eq:vec knot insertion}, it holds that
\[
\hat{\boldsymbol{\beta}}(x,y) = (A_y \otimes A_x)^{T}\,\boldsymbol{\beta}(x,y),
\]
and the tensor-product subdivision matrix $A \in \mathbb{R}^{n \times \hat{n}}$ is given by
\[
A = A_y \otimes A_x.
\]

Additionally, any $\hat{s} \in \hat{\SSS}$ can be expressed in terms of coarse B-spline coefficient matrix $\hat{C} \in \mathbb{R}^{\hat{n}_x \times \hat{n}_y}$:
\[
\hat{s}(x,y) = \vect(\hat{C})^T  \hat{\boldsymbol{\beta}}(x,y) = \vect(\hat{C})^T (A_y \otimes A_x)^{T}\,\boldsymbol{\beta}(x,y) = [(A_y \otimes A_x) \vect(\hat{C})]^T {\boldsymbol{\beta}}(x,y).
\]
Then, the fine B- spline coefficient matrix $C \in \mathbb{R}^{n_x \times n_y}$ satisfies
\[
\vect(C) = (A_y \otimes A_x)\,\vect(\hat{C}),
\qquad\text{or equivalently}\qquad
C = A_x\,\hat{C}\,A_y^{T}.
\]
The coarsening (reverse subdivision) process aims to recover $\hat{C}$ from $C$ using left inverses $B_x$ and $B_y$ of $A_x$ and $A_y$, respectively, i.e.,
\[
B_x A_x = I_{\hat{n}_x},
\qquad
B_y A_y = I_{\hat{n}_y}.
\]
Accordingly, the coarse coefficients are obtained as
\[
\hat{C} = B_x\,C\,B_y^{T},
\qquad\text{or equivalently}\qquad
\vect(\hat{C}) = (B_y \otimes B_x)\,\vect(C).
\]
Finally, using the mixed-product property of the Kronecker product, we obtain
\[
(B_y \otimes B_x)(A_y \otimes A_x)
= (B_y A_y) \otimes (B_x A_x)
= I_{\hat{n}_y} \otimes I_{\hat{n}_x}
= I_{\hat{n}_x \hat{n}_y},
\]
which confirms that $B_y \otimes B_x$ is indeed a left inverse of $A_y \otimes A_x$.

\paragraph{Extension to $D$-directional tensor-product spline spaces.} Let the knot vectors $\hat{\Xi}_1, \dots, \hat{\Xi}_D$ define the coarse tensor-product space $\hat{\SSS}$, and let $\Xi_1, \dots, \Xi_D$ be the refined versions defining the fine space $\SSS$. Denote by $A_d \in \mathbb{R}^{n_d \times \hat{n}_d}$ the univariate subdivision matrices for $d = 1, \dots, D$, and by $B_d \in \mathbb{R}^{\hat{n}_d \times n_d}$ their left inverses.

The tensor-product subdivision matrix is defined as
\[
A = A_D \otimes \cdots \otimes A_1,
\]
and the tensor-product coarsening operator as
\[
B = B_D \otimes \cdots \otimes B_1,
\]
satisfying
\[
BA = (B_D A_D) \otimes \cdots \otimes (B_1 A_1) 
= I_{\hat{n}_1 \cdots \hat{n}_D}.
\]

\section{Experimental analysis}\label{sec:tests}

We conclude this article with a series of numerical experiments that illustrate and further support the theoretical analysis. First, we report and discuss tables of suitably chosen norms that allow us to assess both the stability and the approximation quality of the proposed coarsening operators, considering the cases of univariate splines as well as tensor-product spline spaces. Next, we examine the \emph{optimality} curves of the coarsening process in two dimensions for a selected spline, comparing the approximation errors produced by the proposed operators with the best possible errors obtained through successive uniform coarsenings. Finally, we highlight the advantages that arise from the local nature of the proposed operators.

\subsection{Stability and approximation quality for coarsening operators}

When $\hat{\SSS}$ and $\SSS$ are spline spaces with $\hat{\SSS} \subseteq \SSS$, as discussed in the previous sections, we consider the \emph{inclusion operator} $\mathcal{I} : \hat{\SSS} \to \SSS$, naturally associated with the subdivision matrix $A$. Correspondingly, we define a \emph{coarsening operator} $\mathcal{R} : \SSS \to \hat{\SSS}$, associated with a matrix $B$ constructed in Section~\ref{sec:construction_of_left_inverses} using a fixed locality width. Requiring $B$ to be a left inverse of $A$ is equivalent to imposing that $\mathcal{R}(\hat{s}) = \hat{s}$ for all $\hat{s} \in \hat{\SSS}$.

The $L^\infty$-stability of $\mathcal{R}$ is directly linked to the matrix norm $\|B\|_\infty$. Moreover, thanks to the $L^\infty$-stability of the B-spline basis, the quantity $\|I - AB\|_\infty$ can be interpreted as a measure of the approximation quality of $\mathcal{R}$. Alternatively, since we are working with uniform partitions, the $L^2$-stability of the B-spline basis justifies using $\|I - AB\|_2$ as an indicator of the approximation quality in the $L^2$-norm.

In this numerical study, we evaluate the stability and approximation quality of several coarsening operators. Table~\ref{tab:coarsening1D} reports the values of the aforementioned norms for different locality widths $r$, using different polynomial degrees. The results for tensor-product coarsening operators, as introduced in Section~\ref{sec:tensor product} and using the same locality widths in both parametric directions, are presented in Table~\ref{tab:coarsening2D}. We emphasize that the reported values in both tables remain unchanged when increasing the number of breakpoints in the underlying uniform partitions.

\begin{table}[!ht]
  \centering
  \footnotesize
  \renewcommand{\arraystretch}{1.1}
  \setlength{\tabcolsep}{5pt}

  \begin{minipage}{0.47\textwidth}
    \centering
    \begin{tabular}{ccccc}
      \toprule
      $r$ & $\|B\|_\infty$ & $\|\boldsymbol{\omega}_r\|_2$ & $\|I - AB\|_2$ & $\|I - AB\|_\infty$ \\
      \midrule
      3  & 1.00 & 1.00 & 1.41 & 2.00 \\
      5  & 1.57 & 0.85 & 1.10 & 1.86 \\
      7  & 1.57 & 0.85 & 1.09 & 2.02 \\
      9  & 1.68 & 0.84 & 1.09 & 2.02 \\
      \bottomrule
    \end{tabular}
    \caption*{(a) Linear splines ($p = 1$)}
  \end{minipage}
  \hfill
  \begin{minipage}{0.47\textwidth}
    \centering
    \begin{tabular}{ccccc}
      \toprule
      $r$ & $\|B\|_\infty$ & $\|\boldsymbol{\omega}_r\|_2$ & $\|I - AB\|_2$ & $\|I - AB\|_\infty$ \\
      \midrule
      4  & 2.33 & 1.12 & 1.25 & 1.58 \\
      6  & 2.29 & 1.12 & 1.25 & 1.68 \\
      8  & 2.29 & 1.01 & 1.07 & 1.59 \\
      10 & 2.29 & 1.01 & 1.07 & 1.62 \\
      12 & 2.29 & 1.00 & 1.06 & 1.53 \\
      \bottomrule
    \end{tabular}
    \caption*{(b) Quadratic splines ($p = 2$)}
  \end{minipage}

  \vspace{1em} 

  \begin{minipage}{0.47\textwidth}
    \centering
    \begin{tabular}{ccccc}
      \toprule
      $r$ & $\|B\|_\infty$ & $\|\boldsymbol{\omega}_r\|_2$ & $\|I - AB\|_2$ & $\|I - AB\|_\infty$ \\
      \midrule
      5  & 3.10 & 2.12 & 3.16 & 4.05 \\
      7  & 3.10 & 1.34 & 1.44 & 3.20 \\
      9  & 3.26 & 1.34 & 1.42 & 3.27 \\
      11 & 3.26 & 1.24 & 1.33 & 3.15 \\
      13 & 3.38 & 1.24 & 1.32 & 3.19 \\
      15 & 3.38 & 1.22 & 1.31 & 3.16 \\
      \bottomrule
    \end{tabular}
    \caption*{(c) Cubic splines ($p = 3$)}
  \end{minipage}
  \hfill
  \begin{minipage}{0.47\textwidth}
    \centering
    \begin{tabular}{ccccc}
      \toprule
      $r$ & $\|B\|_\infty$ & $\|\boldsymbol{\omega}_r\|_2$ & $\|I - AB\|_2$ & $\|I - AB\|_\infty$ \\
      \midrule
       6  & 4.75 & 2.23 & 2.30 & 3.25 \\
       8  & 4.75 & 2.23 & 2.30 & 3.25 \\
      10  & 4.53 & 1.66 & 1.41 & 2.84 \\
      12  & 4.48 & 1.66 & 1.40 & 2.86 \\
      14  & 4.48 & 1.54 & 1.31 & 2.68 \\
      16  & 4.46 & 1.54 & 1.31 & 2.70 \\
      18  & 4.46 & 1.51 & 1.29 & 2.59 \\
      \bottomrule
    \end{tabular}
    \caption*{(d) Quartic splines ($p = 4$)}
  \end{minipage}
  \caption{Univariate splines: Stability and accuracy of coarsening operators built with different locality widths $r$ for different polynomial degrees. The reported values are independent of the partition size used.}
  \label{tab:coarsening1D}
\end{table}
  
\begin{table}[!hb]
\centering
\footnotesize
\renewcommand{\arraystretch}{1.1}
\setlength{\tabcolsep}{5pt}

\begin{subtable}[t]{0.48\textwidth}
  \centering
  \begin{tabular}{@{}cccc@{}}
    \toprule
    $r$ & $\|B\|_\infty$ & $\|I - AB\|_2$ & $\|I - AB\|_\infty$ \\
    \midrule
    3 & 1.00  & 1.98 & 2.00 \\
    5 & 2.47  & 1.22 & 2.61 \\
    7 & 2.47  & 1.18 & 2.85 \\
    9 & 2.83  & 1.18 & 3.00 \\
    \bottomrule
  \end{tabular}
  \caption{Bilinear splines ($p = 1$).}
  \label{tab:bilinear}
\end{subtable}
\hfill
\begin{subtable}[t]{0.48\textwidth}
  \centering
  \begin{tabular}{@{}cccc@{}}
    \toprule
    $r$ & $\|B\|_\infty$ & $\|I - AB\|_2$ & $\|I - AB\|_\infty$ \\
    \midrule
    4  & 5.44 & 1.55 & 3.12 \\
    6  & 5.25 & 1.55 & 3.14 \\
    8  & 5.25 & 1.15 & 2.95 \\
    10 & 5.23 & 1.14 & 2.90 \\
    12 & 5.23 & 1.13 & 2.83 \\
    \bottomrule
  \end{tabular}
  \caption{Biquadratic splines ($p = 2$).}
  \label{tab:biquad}
\end{subtable}

\vspace{1em}

\begin{subtable}[t]{0.48\textwidth}
  \centering
  \begin{tabular}{@{}cccc@{}}
    \toprule
    $r$ & $\|B\|_\infty$ & $\|I - AB\|_2$ & $\|I - AB\|_\infty$ \\
    \midrule
    5  & 9.62  & 9.94 & 10.19 \\
    7  & 9.62  & 2.06 & 5.86 \\
    9  & 10.64 & 2.01 & 6.26 \\
    11 & 10.64 & 1.76 & 5.97 \\
    13 & 11.40 & 1.75 & 6.20 \\
    15 & 11.40 & 1.71 & 6.11 \\
    \bottomrule
  \end{tabular}
  \caption{Bicubic splines ($p = 3$).}
  \label{tab:bicubic}
\end{subtable}
\hfill
\begin{subtable}[t]{0.48\textwidth}
  \centering
  \begin{tabular}{@{}cccc@{}}
    \toprule
    $r$ & $\|B\|_\infty$ & $\|I - AB\|_2$ & $\|I - AB\|_\infty$ \\
    \midrule
     6  & 22.56 & 4.96 & 11.77 \\
     8  & 22.56 & 4.84 & 11.94 \\
    10  & 20.55 & 2.01 & 7.84 \\
    12  & 20.11 & 2.24 & 7.88 \\
    14  & 20.11 & 1.71 & 7.26 \\
    16  & 19.90 & 2.22 & 7.35 \\
    18  & 19.90 & 2.22 & 6.89 \\
    \bottomrule
  \end{tabular}
  \caption{Biquartic splines ($p = 4$).}
  \label{tab:biquartic}
\end{subtable}

\caption{Tensor-product splines: Stability and accuracy of coarsening operators built with different locality widths $r$ for different polynomial degrees. The reported values are independent of the mesh size used.}
\label{tab:coarsening2D}
\end{table}

\subsection{Successive coarsening in tensor product spline spaces}

We assess the performance of the proposed coarsening operators for tensor-product spline spaces by applying them to a non-spline function. Specifically, we consider
\[
f(x, y) = \arctan\left(5\left[(4x - 3.5)^2 + (4y - 3)^2 - 5\right]\right).
\]
For each polynomial degree \(p = 1, 2, 3, 4\), we select two symmetric values of the coarsening locality width, denoted by \(r\), corresponding to the case \((r_x, r_y) = (r, r)\). We remark that the larger value of \(r\) is chosen to ensure that the resulting operator behaves almost identically to those with greater locality.

The procedure is as follows. The function \(f\) is first approximated in the reference fine space using the \(L^2\)-projection. The resulting fine-level B-splines coefficients are then iteratively coarsened using the proposed local operators, and at each coarsening step the \(L^2\)-error with respect to \(f\) is computed in the corresponding coarse space. These errors are compared against those obtained via the standard \(L^2\)-projection onto the same coarse spaces.

The results are displayed in Figure~\ref{fig:coarsening_errors_all_degrees}, where the \(L^2\)-error is plotted against the number of degrees of freedom (DOFs) for the different polynomial degrees \(p\) and the two symmetric locality widths \((r, r)\) under consideration. 
For $p =1, 2$, both coarsening operators yield approximation errors comparable to those of the optimal $L^2$-projection. For $p=3,4$, however, the operator associated with the smaller locality width $r$ produces noticeably larger errors than those obtained with the larger value of $r$. Importantly, the errors produced by the operator with the larger $r$ is already very close to the $L^2$-projection. These observations indicate that increasing the locality width $r$ improves accuracy while preserving the intrinsic locality of the operator.

\begin{figure}[htbp]
    \centering
    \begin{minipage}{0.48\textwidth}
        \centering
        \includegraphics[width=\linewidth]{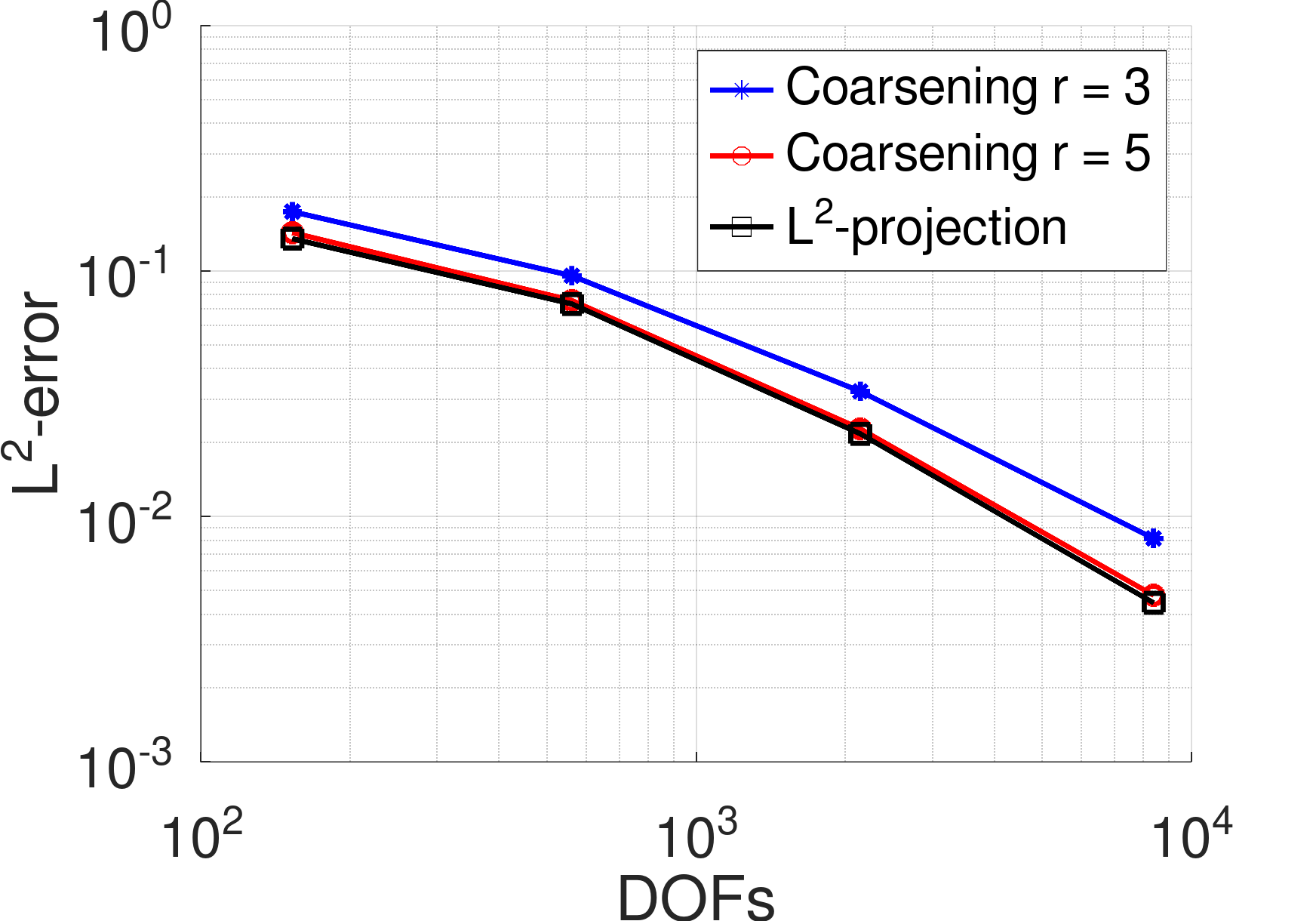}
        \caption*{(a) Bilinear splines ($p = 1$)}
    \end{minipage}
    \hfill
    \begin{minipage}{0.48\textwidth}
        \centering
        \includegraphics[width=\linewidth]{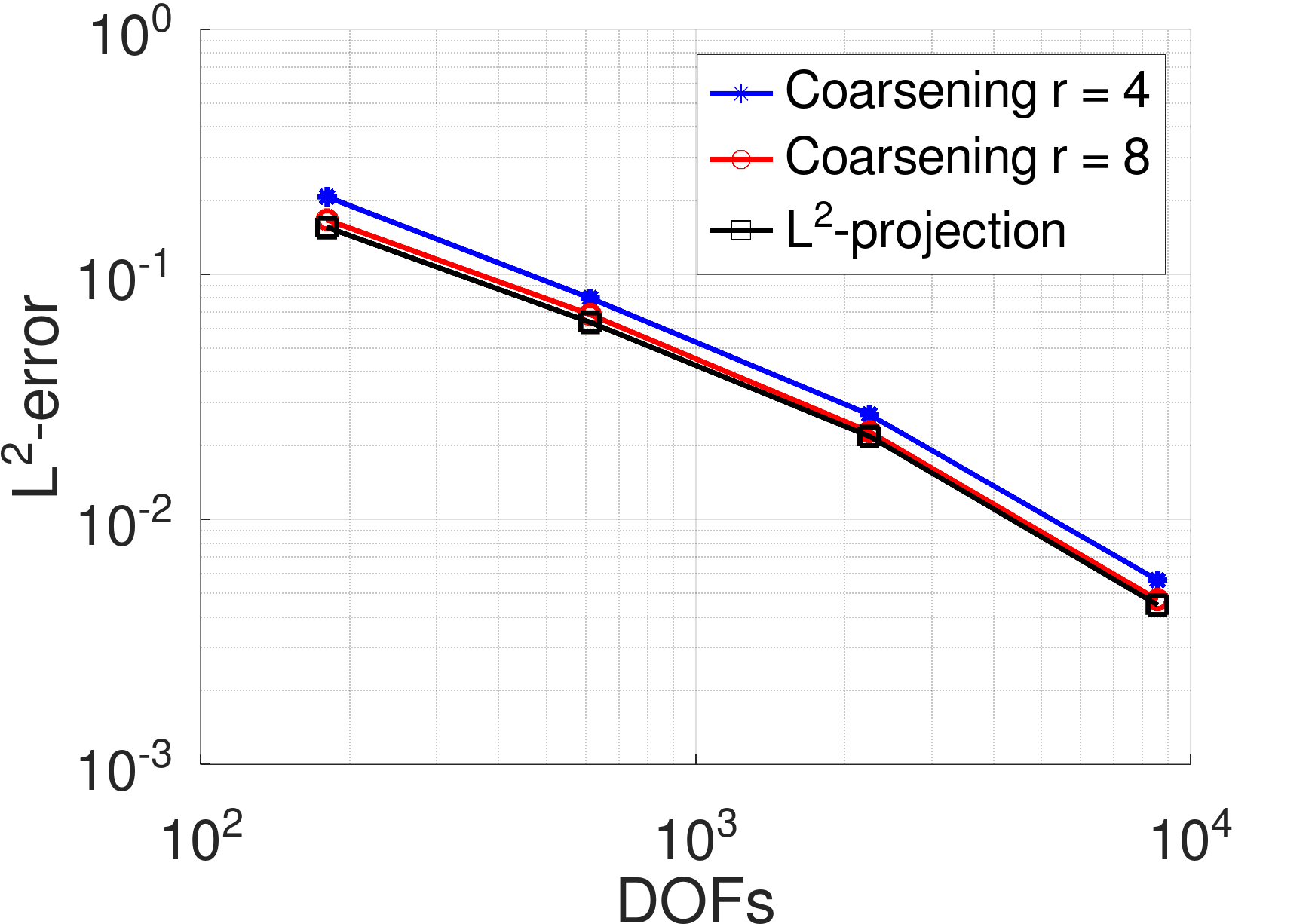}
        \caption*{(b) Biquadratic splines ($p = 2$)}
    \end{minipage}
    
    \vspace{4mm}
    
    \begin{minipage}{0.48\textwidth}
        \centering
        \includegraphics[width=\linewidth]{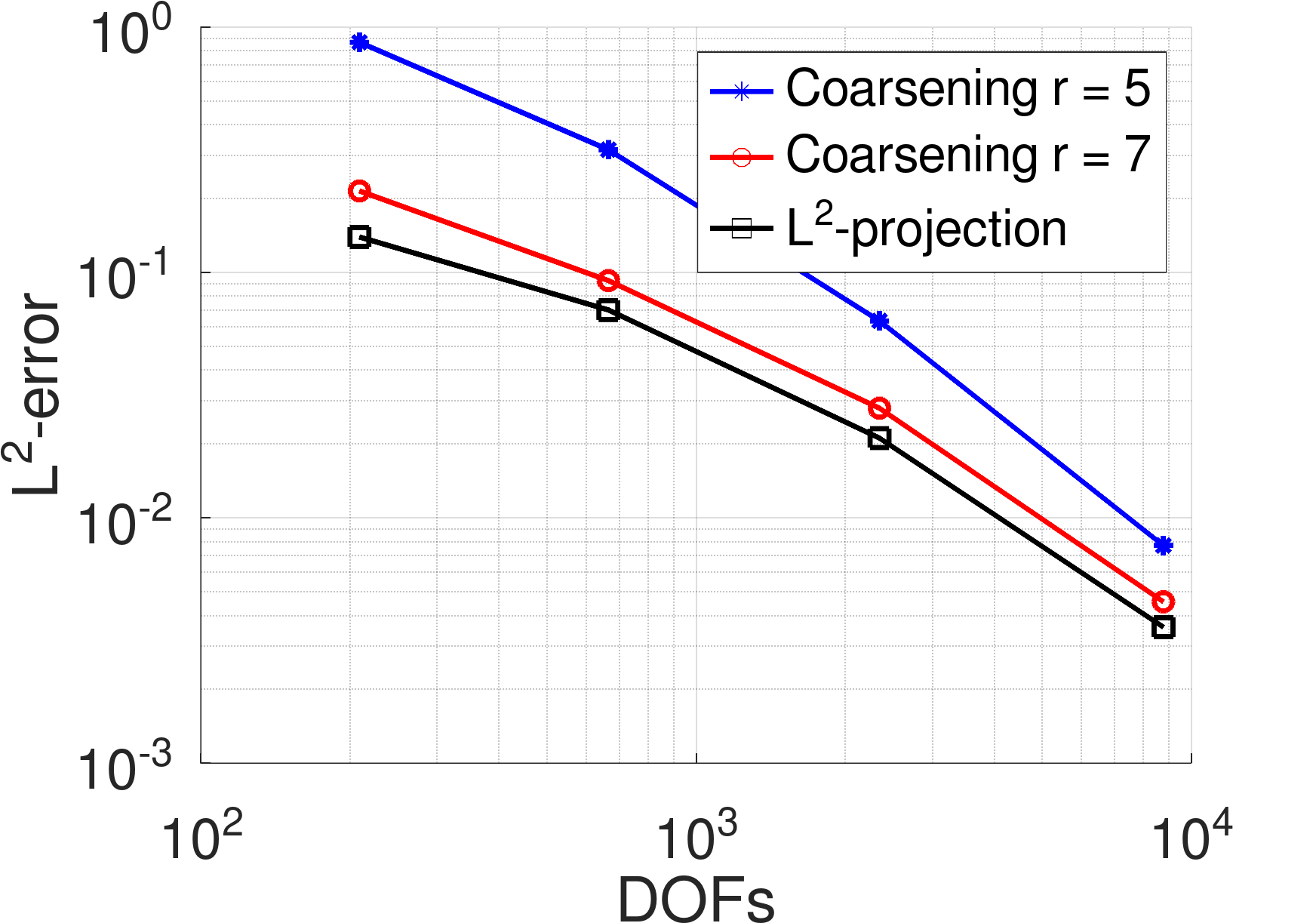}
        \caption*{(c) Bicubic splines ($p = 3$)}
    \end{minipage}
    \hfill
    \begin{minipage}{0.48\textwidth}
        \centering
        \includegraphics[width=\linewidth]{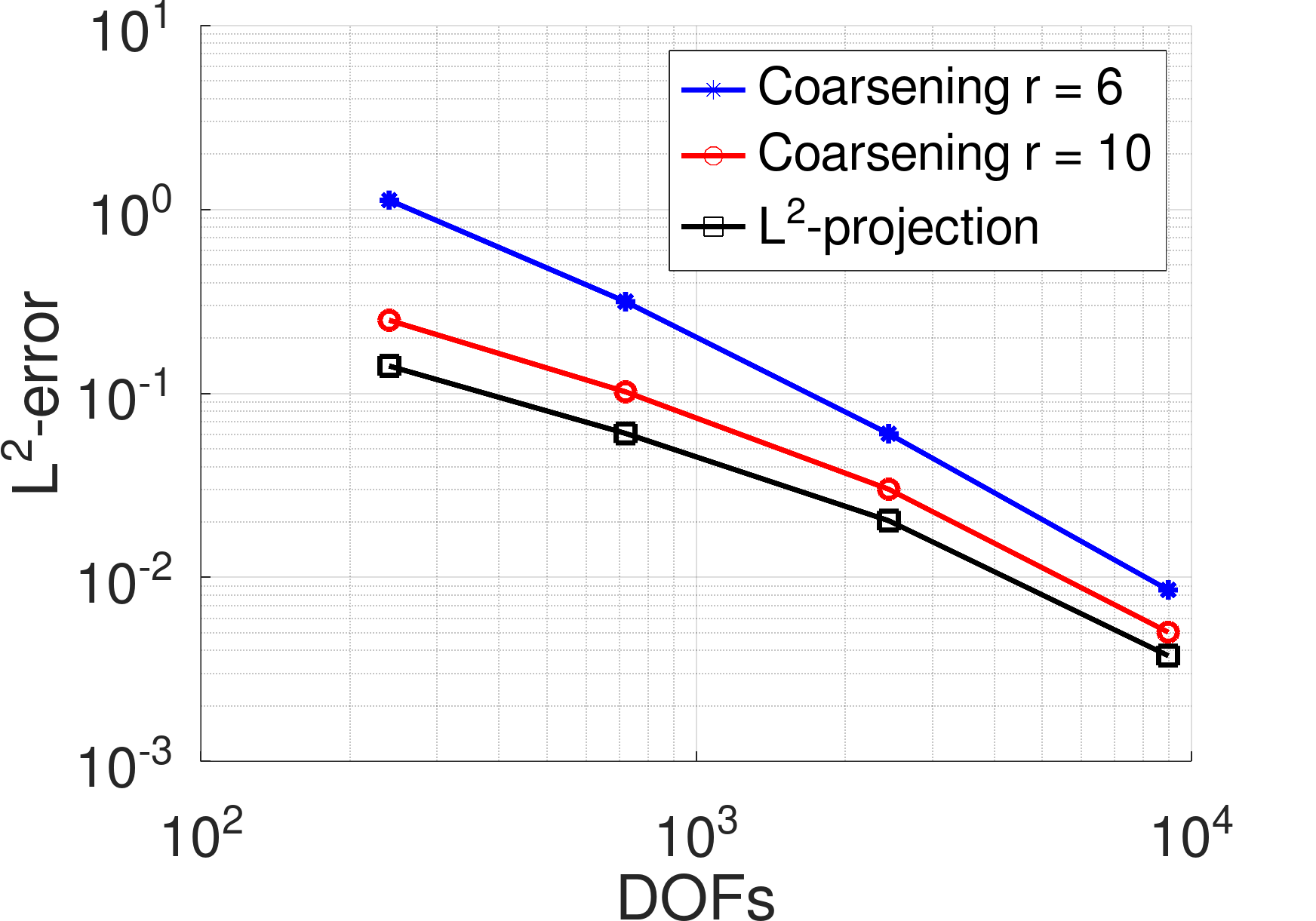}
        \caption*{(d) Biquartic splines ($p = 4$)}
    \end{minipage}
    
    \caption{
        Comparison of $L^2$-approximation errors for a function $f(x, y)$ computed using the proposed local coarsening operators and the standard $L^2$-projection.
        Results are shown for spline degrees $p = 1, 2, 3, 4$ (subfigures (a)–(d), respectively).
        In each case, two symmetric locality widths $(r_x, r_y) = (r, r)$ are considered. The curves indicate that increasing \(r\) systematically improves the approximation quality while preserving the local nature of the operator: in all cases the largest of the two locality widths yields errors that are almost indistinguishable from those obtained by the $L^2$-projection.}
    \label{fig:coarsening_errors_all_degrees}
\end{figure}

\subsection{Single uniform coarsening for a tensor product spline}

We illustrate the performance of the proposed coarsening operators for tensor-product spline spaces by considering a spline function that exhibits a sharp jump along a prescribed path in the domain. More precisely, the B-spline coefficients of the original spline are all zero except for those associated with basis functions in a localized region, which are set to one. In such a scenario, uniform coarsening is expected to produce large approximation errors, since even the standard \(L^2\)-projection onto the coarse mesh displays significant oscillations. In contrast, the coarsening operators studied here, in addition to being considerably faster to apply due to their local character, also have the desirable property of localizing the error: oscillations remain concentrated near the region where the coarse space cannot exactly represent the spline, while the function is preserved elsewhere.

In Figure~\ref{fig:coarsening_deg2}, we consider a biquadratic spline space defined over a $40 \times 40$ element mesh. We present the approximations obtained with the $L^2$-projection and with two coarsening operators, both applied using equal locality widths in each direction, namely $6$ and $8$. The rightmost column displays the sparsity pattern of the B-spline coefficient matrix of the error measured in the fine space. One can observe that the $L^2$-projection modifies all coefficients of the original spline, whereas the coarsening operators modify only $18\%$ and $32\%$ of the coefficients, respectively, when expressed in the fine space. In this example, the relative $L^\infty$-error for width $6$ is about $40\%$, slightly larger than the $29\%$ obtained with the $L^2$-projection. With width $8$, however, the error decreases to $27\%$, comparable to that of the projection.

\begin{figure}[htbp]
  \centering
  \includegraphics[width=\textwidth]{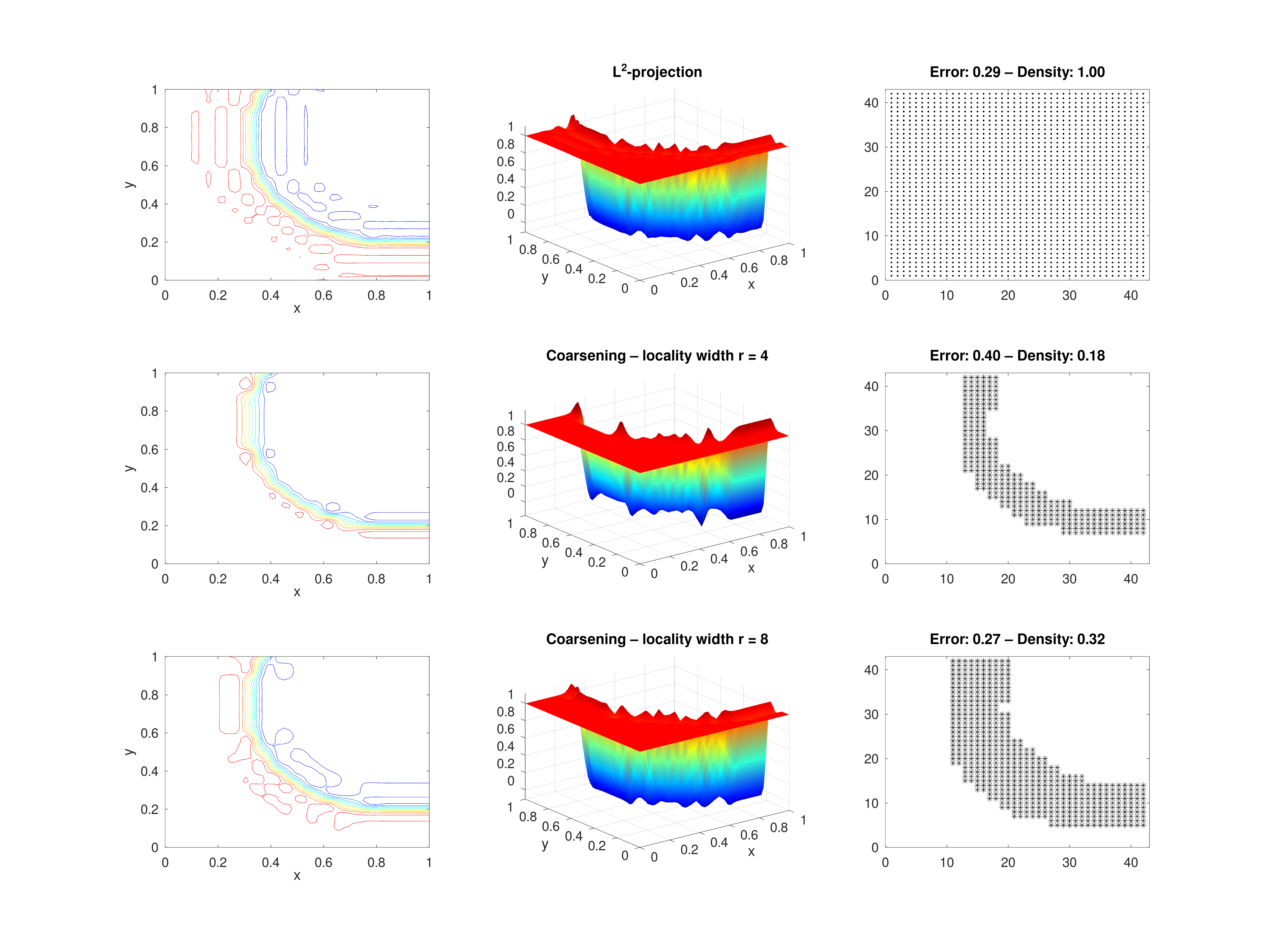}
  \caption{
  Considering a biquadratic spline defined on a tensor-product space with an initial mesh of \(40 \times 40\) elements, we compare three approximations computed on a coarser \(20 \times 20\) mesh. The first row corresponds to the standard \(L^2\)-projection, while the second and third rows show the results obtained with the proposed coarsening operators using locality widths of \(4\) and \(8\) in each parametric direction, respectively. For each case, the left panel shows the contour lines of the approximation, the middle panel displays the 3D surface, and the right panel illustrates the sparsity pattern of the matrix containing the B-spline coefficients of the error. The relative error in the \(L^\infty\)-norm of each approximation is also reported.
  }
  \label{fig:coarsening_deg2}
\end{figure}

In Figure~\ref{fig:coarsening_deg3}, we turn to the case of bicubic splines, again defined over a $40 \times 40$ grid. We compare the $L^2$-projection with two coarsening operators, this time using locality widths of $5$ and $7$. For width $5$, the approximation modifies only a very localized region, but noticeable oscillations appear, leading to reduced overall quality. With width $7$, on the other hand, the approximation reaches a level comparable to the $L^2$-projection while still altering only a relatively small fraction of coefficients, which is advantageous in practice. These observations are in line with the first two rows of Table~\ref{tab:bicubic}.

\begin{figure}[htbp]
  \centering
  \includegraphics[width=\textwidth]{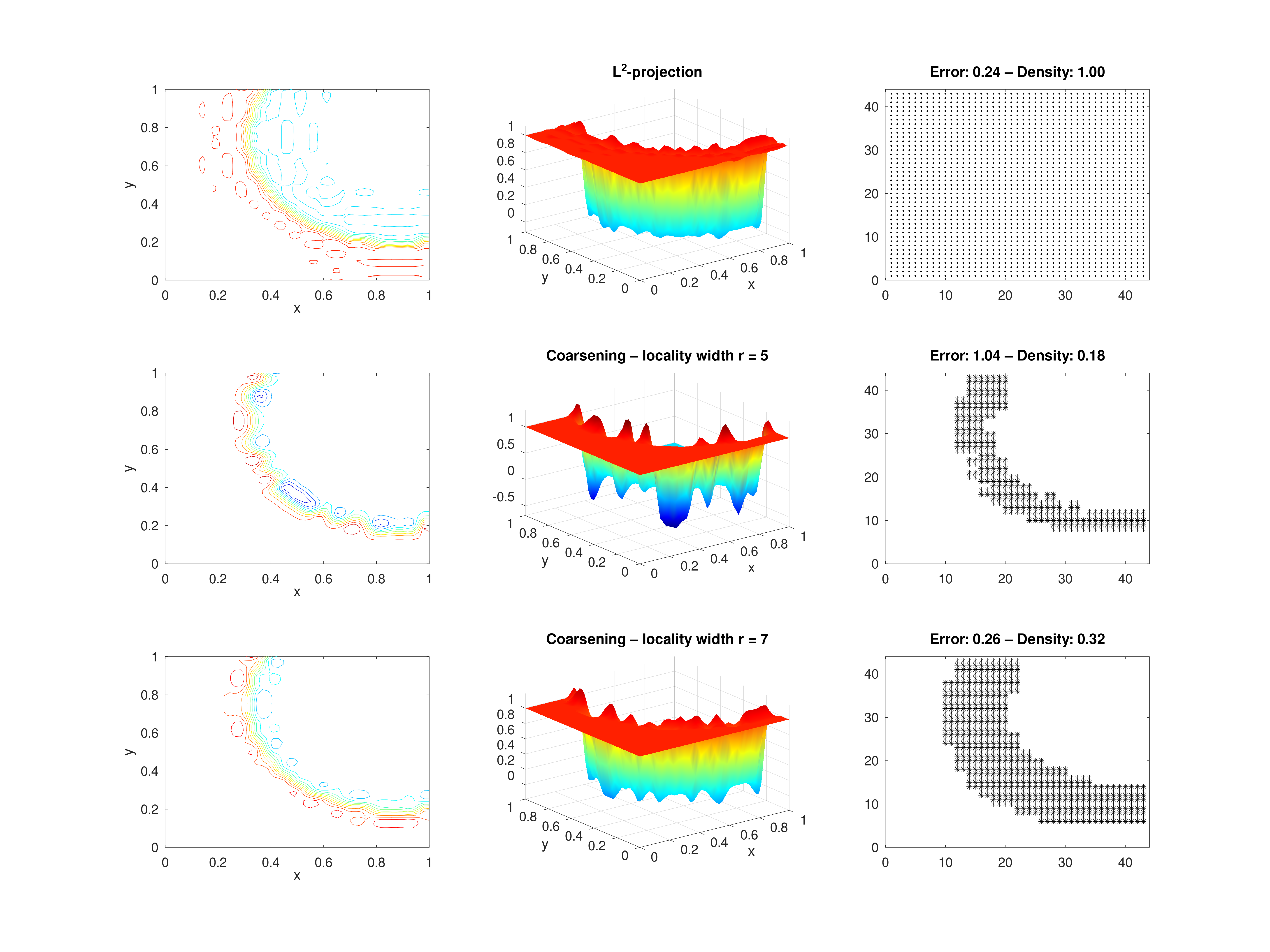}
  \caption{
  Considering a bicubic spline defined on a tensor-product space with an initial mesh of \(40 \times 40\) elements, we compare three approximations computed on a coarser \(20 \times 20\) mesh. The first row corresponds to the standard \(L^2\)-projection, while the second and third rows show the results obtained with the proposed coarsening operators using locality widths of \(5\) and \(7\) in each parametric direction, respectively. For each case, the left panel shows the contour lines of the approximation, the middle panel displays the 3D surface, and the right panel illustrates the sparsity pattern of the matrix containing the B-spline coefficients of the error. The relative error in the \(L^\infty\)-norm of each approximation is also reported.
  }
  \label{fig:coarsening_deg3}
\end{figure}

\section*{Conclusions and future work}

In this work we have introduced and analyzed a class of local coarsening operators for spline spaces, constructed as left inverses of the standard subdivision operators. The proposed approach guarantees exact reproduction on the coarse space while providing stability and controllable approximation quality, as demonstrated through a combination of theoretical results and numerical experiments. In particular, the local nature of the operators makes them computationally efficient and capable of confining approximation errors to the regions where exact representation is not possible, which is an advantage over global procedures such as the $L^2$-projection.

The presented framework opens several directions for further research. One promising direction concerns the integration of the proposed operators into multilevel algorithms, for instance in the context of hierarchical spline constructions. Another natural extension is the study of adaptive coarsening strategies, where the locality width is chosen according to local error indicators. Finally, the analysis of coarsening for more general spline configurations, such as non-uniform meshes, remains an interesting subject for future~investigation.

\bibliographystyle{apalike}

	\end{document}